\documentclass[11pt]{article}\textwidth 160mm\textheight 235mm
\usepackage[table,xcdraw]{xcolor}
\usepackage{amsfonts}
\usepackage{amssymb}
\usepackage{graphicx}
\usepackage{amsmath}
\usepackage{amsthm}
\usepackage{enumitem}
\usepackage{tikz}
\usepackage{float}
\usepackage{cancel}
\usepackage{todonotes}
\usepackage{tcolorbox}
\tcbuselibrary{skins}
\usepackage{lipsum}
\usetikzlibrary{matrix}
\usetikzlibrary{arrows,shapes}
\usepackage{cite}
\usepackage{hyperref}
\hypersetup{colorlinks=true, urlcolor=blue}
\expandafter\let\expandafter\oldproof\csname\string\proof\endcsname
\let\oldendproof\endproof
\renewenvironment{proof}[1][\proofname]{%
  \oldproof[\ttfamily \scshape \bf #1. ]%
}{\oldendproof}
\def\O{{\cal O}}

\def\B{\mathbb{B}}
\def\R{{\rm I\!R}}
\def\N{{\rm I\!N}}
\def\ox{\bar{x}}
\def\oy{\bar{y}}

\def\P{\mathbb{P}}
\def\ss{\scriptsize }

\def\X{{\mathbb X}}

\def\dist{{\rm dist}}

\def\d{{\rm d}}

\def\sub{\partial}

\def\disp{\displaystyle}

\def\Bar{\overline}
\def\ra{\rangle}
\def\la{\langle}

\def\co{\mbox{\rm co}\,}
\def\cone{\mbox{\rm cone}\,}

\def\epi{\mbox{\rm epi}\,}

\def\dom{\mbox{\rm dom}\,}

\def\proj{\mbox{\rm proj}\,}

\def\e{\mbox{\rm e}\,}

\def\Lim{\mbox{\rm Lim}\,}

\def\dn{\downarrow}
\def\O{\Omega}

\def\st{\stackrel}
\def\oR{\Bar{\R}}

\def\dd{\delta}
\def\al{\alpha}

\def\sm{\hbox{${1\over 2}$}}

\def\sce{\setcounter{equation}{0}}

\begin{document}
\vspace*{0.5in}
\begin{center}
{\bf FIRST-ORDER VARITIONAL ANALYSIS OF NON-AMENABLE COMPOSITE FUNCTIONS }\\[1 ex]
ASHKAN MOHAMMADI \footnote{Department of Mathematics and Statistics, Georgetown University, Washington, DC 20057, USA (ashkan.mohammadi@georgetown.edu).}

\end{center}
\vspace*{0.05in}
\small{\bf Abstract.} This paper is devoted to studying the first-order variational analysis of non-convex and non-differentiable functions that may not be subdifferentially regular. To achieve this goal, we entirely rely on two concepts of directional derivatives known as subderivative and semi-derivative. We establish the exact chain and sum rules for this class of functions via these directional derivatives. These calculus rules provide an implementable auto-differentiation process such as back-propagation in composite functions. The latter calculus rules can be used to identify the directional stationary points defined by the subderivative. We show that the distance function of a geometrically derivable constraint set is semi-differentiable, which opens the door for designing first-order algorithms for non-Clarke regular constrained optimization problems. We propose a first-order algorithm to find a directional stationary point of non-Clarke regular and perhaps non-Lipschitz functions. We introduce a descent property under which we establish the non-asymptotic convergence of our method with rate $O(\varepsilon^{-2})$, akin to gradient descent for smooth minimization. We show that the latter descent property holds for free in some interesting non-amenable composite functions, in particular, it holds for the Moreau envelope of any bounded-below function.\\[1ex]


\newtheorem{Theorem}{Theorem}[section]
\newtheorem{Proposition}[Theorem]{Proposition}
\newtheorem{Remark}[Theorem]{Remark}
\newtheorem{Lemma}[Theorem]{Lemma}
\newtheorem{Corollary}[Theorem]{Corollary}
\newtheorem{Definition}[Theorem]{Definition}
\newtheorem{Example}[Theorem]{Example}
\newtheorem{Algorithm}[Theorem]{Algorithm}
\renewcommand{\theequation}{{\thesection}.\arabic{equation}}
\renewcommand{\thefootnote}{\fnsymbol{footnote}}

\normalsize

\section{Introduction}\sce
In this paper, we consider the composite optimization problem
\begin{equation}\label{comp}
\mbox{minimize} \;\;  (g \circ F) (x) \quad \mbox{subject to}\;\; x \in \R^n,
\end{equation}
where $g:\R^m \to \oR :=(-\infty , + \infty]$ is lower semicontinuous, and $F: \R^n \to \R^m$ is semi-differentiable in the point of question; see section \ref{sect01} for definition of semi-differentiability. Problem \eqref{comp} covers constrained optimization problems when $g$ is an extended-real-valued function. Indeed, by redefining $g$ and $F$ as $g(y) + \delta_{X}(z)$ and $(F(x) ,  G(x))$, problem \eqref{comp} boils down to the constrained optimization problem
\begin{equation}\label{ccomp}
\mbox{minimize} \;\;  f(x):= (g \circ F) (x) \quad \mbox{subject to}\;\; G(x) \in X
\end{equation}
where $ G : \R^n \to \R^k$,  and where $X$ is a closed (possibly non-convex) set in $\R^k$. The broadness of problems \eqref{comp} and \eqref{ccomp} allows us to cover many optimization problems, in particular, those with amenable composite functions; see \cite{bf, lw, zl, l, dp, ddmp, dil, r21}. Recall from \cite[Definition~10.23]{rw} that the composite function $f =g \circ F$ is called amenable if $g$ is convex, $F$ is continuously differentiable, and a constraint qualification is present. In this paper, we are specially interested in non-amenable composite function which have not been investigated extensively. More precisely, the composite function $g \circ F$ is called non-amenable if either $g$ is non-convex or $F$ is non-differentiable. Our main objectives in this paper are
\begin{itemize}
\item[•]to obtain calculus rules for subderivative and semi-derivative,
\item[•]to characterize a suitable notion of stationary point via the established calculus rules,
\item[•]to design an algorithm by which we can compute such stationary points.
\end{itemize}
\subsection{Our Contribution}

Our approach to handle the composite optimization problem \eqref{comp} is purely primal which allows us to avoid Lagrangian multipliers. More precisely, we use subderivative as a generalized directional derivative \cite[p.~257]{rw}, that is for a given function $f:\R^n \to \oR = (- \infty , \infty]$ and a point $\ox$ with $f(\ox)$ finite, the subderivative function $\d f(\ox)\colon\R^n\to[-\infty,\infty]$ is defined by
\begin{equation}\label{subderivatives}
{\mathrm d}f(\ox)(w):=\liminf_{\substack{
   t\dn 0 \\
  w' \to w
  }} {\frac{f(\ox+tw')-f(\ox)}{t}}.
\end{equation}
Note that when $f$ is differentiable at $\ox$, we have ${\mathrm d}f(\ox)(w) = \la \nabla f(\ox) , w \ra .$ We establish a chain rule for the subderivative of composite function \eqref{comp}. Utilizing subderivative, we propose a generalized gradient descent method, where we call it the subderivative method: 
\vspace{.5cm}
 \begin{table}[H]
        \renewcommand{\arraystretch}{1}
        \centering
        \label{alg1}
        \begin{tabular}{|
                p{0.9\textwidth} |}
            \hline
            \textbf{Subderivative Method}\\
 
            \begin{enumerate}[label=(\alph*)]
            
                \item[0.] \textbf{(Initialization)} Pick the tolerance $\epsilon \geq 0$, a starting point $x_0 \in \R^n$, and set $k=0 .$
           
                \item[1.] \textbf{(Termination)} Stop, if a prescribed stopping criterion is satisfied.
                \item[2.] \textbf{(Direction Search)} Pick $w_k \in \mbox{arg min}_{\|w\| \leq 1 }   \d f(x_k)(w).$
                \item[3.] \textbf{(Line Search)} Choose a step size $\alpha_{k}  > 0$ through a line search method.
                \item[4.] \textbf{(Update)} Set $x_{k+1} : = x_k + \alpha_k  w_k$ and $k+1  \leftarrow k$ and then go to step 1. 
            \end{enumerate}
            \\ \hline
        \end{tabular}
    \end{table}
\vspace{.5cm}    
The subderivative method reduces to the classical gradient descent method if $f$ is differentiable. Indeed, if $x_k$ is a differentiable point of $f$, with $ \nabla f(\ox) \neq 0$, the direction search in the subderivative method is $w_k =  - \frac{ \nabla f(\ox)}{\| \nabla f(\ox)\|}$. We establish a convergence of the subderivative method with the Armijo backtracking line search under the following condition
\begin{equation*}\label{descet_pro}
f(y) \leq f(x) + \d f(x) (y-x) + \frac{L}{2}  \| y -x  \|^2
 \quad \quad x,y \in \R^n  .\end{equation*}
Notice that when $f$ is $L$ -smooth, that is $\nabla f$ is Lipschitz continuous with constant $L$, the above property follows from the classical descent lemma; see \cite{N}. Therefore, the subderivative method together with its convergence analysis reduces to the classical gradient descent method for $L$-smooth functions. The rest of the paper is organized as follows. In section \ref{sect01} we mention some examples in the framework of problems \eqref{comp} or \eqref{ccomp}, then briefly point out the related works for each of them. Section \ref{sect02} is for the variational analysis of problem \eqref{comp}. In particular, we establish calculus rules for subderivative to identify the directional stationary points of problem \eqref{comp} and consequently obtain the directional stationary points for all examples in section \ref{sect01}. Finally, in section \ref{sec03}, we propose the subderivative method to find an directional stationary point of problem \eqref{comp} with some positive tolerance.   
\section{Examples}\sce\label{sect01}

Below we mention several examples of non-amenable composite optimization problems which fall into framework \eqref{comp}. The generalized calculus established in section \ref{sect02} enables us to compute the subderivative of the essential function of each example. Nevertheless, each example requires full treatment of numerical analysis for (approximately) solving the subderivative method's sub-problems. The latter remains for our future work; in \cite{m22}, where we give a full treatment of the numerical analysis of Example \ref{dmax}.

\begin{Example}[\bf training a neural network problem with the ReLU activation]\label{tnn}{\rm
Fix positive integers $N, n_i \in \N$ for $i=1,...,N$. Let $W^{i} \in \R^{n_{i-1} \times n_{i}} $ and $b^{i} \in \R^{n_i}$ for $i=1,2,..N$ stand for the wieghts and biases which are yet to be determined. Fix functions the vector valued functions $\sigma^{i} : \R^{n_i} \to \R^{n_{i}}$ for $i=1,2,.., N-1$ for which the j-th component of each $\sigma^i$ is the ReLU activation function acting on the j-th component of its input, i.e.,  $[\sigma^i]_{j}(x) = \max\{ 0 , x_j\}$. Then, the fully connected ReLU neural network function with $N+1$ layers ( $N-1$ hidden layers) is defined by the following function ($x$ is the input variable)
\begin{equation}\label{nnf}
f(W, b \; ; \; x) =  W^N  \sigma^{N-1} \big( . \; . \; .  \big( W^2 \sigma^1 (W^1  x  - b^1) - b^2\big) . \; . \; . \big)-b^N  
\end{equation}
where $W =(W^1,...,W^N )$ and $b=(b^1,...,b^N )$. By training a neural network, one means to find the weights $W^i$ and biases $b^i$ such that the function \eqref{nnf} approximately fits a given data set of pairs $(x,y)$. Indeed, by taking $\{ (x^i , y^i) \in \R^{n_0} \times \R^{n_{N}} \; | \; i=1,2,...M \}$ as a set of (training) data, the goal is to minimize the following \textit{loss} function 
\begin{equation}\label{l2los}
\mbox{minimize} \;\;  L(W, b) := \frac{1}{M}\; \sum_{i=1}^{M} \|f(W,b  \; ; \; x^i) \; - \; y^i\|^2  \quad \mbox{over all}\;\;   W  \;, \; b
\end{equation}}
\end{Example}

\begin{Example}[\bf bilevel programming]\label{bilevel}{\rm In this example, we consider a bilevel programming in the following form:
\begin{equation*}\label{bp}
(BP) \quad \quad \mbox{minimize} \;\;  \varphi (x,y)  \quad \mbox{subject to}\;\;    y \in S(x)
\end{equation*}
where, for any given $x$, $S(x)$ denotes the solution set of the lower-level program
\begin{equation*}\label{px}
(P_x) \quad \quad \mbox{minimize} \;\;  \psi (x,y)  \quad \mbox{subject to}\;\;   h(x,y) \leq 0
\end{equation*}
and $\varphi, \psi : \R^n \times \R^m \to \R$, $h:\R^n \times \R^m \to \R^k$ are continuously differentiable. Similar to \cite{by}, one can place more functional constraints in (BP). For simplicity, we do not add more constraints to (BP). A common approach to obtain optimality conditions and solving (BP) numerically, is to reformulate it into a single-level optimization problem, then applying the optimality condition on the latter problem. In the introduction of \cite{by}, there is a comprehensive discussion about different ways of single-level reformulation of (BP). In this paper, we study the value function approach proposed in Outrata \cite{o}. More precisely, in the sense of global minimizers, (BP) is equivalent to the following single-valued problem:
\begin{equation*}\label{v}
(VP) \quad \quad \mbox{minimize} \;\;  \varphi (x,y)  \quad \mbox{subject to}\;\;   \psi(x,y)- V(x) \leq 0, ~ h(x,y)\leq 0
\end{equation*}         
where, $V(x) := \inf \{ \psi(x,y)|\; h(x,y)\leq 0 \}$ is the optimal value function of the lower-level problem $P_x$. The value function $V$ is likely non-differentiable, causing problem (VP) not fitting into the amenable framework. However, under some reasonable assumptions, e.g. \cite[Theorem~5]{mins}, the value function is locally Lipschitz continuous and directionally differentiable, hence, it is semi-differentiable at the point of question.  
}
\end{Example}
 
\begin{Example}[\bf mathematical programming with complementary constraints]\label{complec}{\rm
Let $K , H : \R^n \to \R^m$ be continuously differentiable functions. Consider the optimization problem
\begin{equation}\label{complec1}
\mbox{minimize} \;\;  f(x)  \quad \mbox{subject to}\;\;    0 \geq K(x) \perp H(x)  \leq 0.
\end{equation}
It is well-know that, due to the complementarity constraint, the Mangasarian–Fromovitz constraint qualification never holds at any point of the feasible set; see \cite[Proposition~1.1]{jzz}. However, a suitable reformulation of the optimization problem \eqref{complec1} may guaranteee  a constraint qualification. Namely, the optimization problem \eqref{complec1} falls into framework \eqref{ccomp} by setting $G(x) := (H(x) , K(x))$ and  $$ X := \{ (y ,z) \in \R^{2k} \; | \; \la x ,y\ra = 0 ,~ x, y \in  \R^k_{-}   \} $$
This reformulation allows us to take advantage of reasonable constraint qualifications, e.g. metric subregularity, to not only drive a necessary optimality condition but also remove the constraint set by an exact penalty. It is worth mentioning that the latter reformulation of \eqref{complec1} is non-amenable, since $X$ is neither convex. More generally, if $C$ is a closed convex cone in the Euclidean space $E$,  which is equipped with the inner product $\la . \ra $, the following optimization problem falls into framework \eqref{comp} by similar reformulation
\begin{equation}\label{complec2}
\mbox{minimize} \;\;  f(x)  \quad \mbox{subject to}\;\;    C \ni  K(x) \perp H(x)  \in C.
\end{equation}
}
\end{Example}
\begin{Example}[\bf difference of amenable functions]\label{dame}{\rm
Another class of optimization problems that can follow the framework \eqref{comp} is the difference of convex functions, which are likely non-amenable if the concave part is non-smooth. Difference of convex functions can be considered as a sub-class of  difference of amenable functions, which are the functions written in the form
$$ f(x) := g_1 \circ F_1 (x)    -  g_2 \circ F_2 (x) $$   
where each $g_i$ is convex and each $F_i$ is continuously differentiable. Particularly, in \cite{m22}, we investigated the following class of optimization problems which is a far (non-convex) extension of difference of max of convex functions which was studied in \cite{pra}.}
\end{Example}  
\begin{Example}[\bf difference of max functions]\label{dmax}{\rm Let $\varphi : \R^n \to \R$ be an amenable function and $f_i : \R^n \to \R$ be differentiable functions for $i=1,2,...,m.$ Then, the following optimization problem falls into framework \eqref{comp}. 
\begin{equation}\label{dmax1}
\mbox{minimize} \;\;  \varphi(x) - \max \{ f_1 (x), f_2 (x), \;  ... \; , f_m (x) \}  \quad \mbox{subject to}\;\;    x \in \R^n
\end{equation}
The optimization problem \eqref{dmax1} can be equivalently written as 
\begin{equation}\label{dmax2}
\mbox{minimize} \;\;  \min \{ \varphi(x) - f_1 (x), \; ... \; , \varphi(x) - f_m (x) \}  \quad \mbox{subject to}\;\;    x \in \R^n .
\end{equation}
In many applications, $m$ is large ; see \cite{asp, pra} and references therein. In the latter case, the pointwise min produces many (subdifferentially) non-regular points, thus the Clarke subdifferential may be unnecessarily huge and provide misleading information.
}
\end{Example}

\begin{Example}[\bf spars optimization]\label{spars}{\rm Let $\O \subseteq \R^n$ bw non-empty set. Finding the sparsest vector in $A \O + b$ can be formulated by the following optimization problem 
\begin{equation}\label{spars1}
\mbox{minimize} \;\;  \| Ax + b \|_{0}  \quad \mbox{subject to}\;\;    x \in \O
\end{equation}
where $\| y \|_0$ counts the number of non-zero components of $y$.  If $\O$ is convex, then the typical approach to solve the optimization problem \eqref{spars1} is to replace $\| . \|_{0}$ with  $\| . \|_1$ as a surrogate function. However, the $\ell_1$ surrogate may not close enough to the actual $\| . \|_0$, thus other types of sparsities can be considered as well; see \cite{cp}.   
}
\end{Example}
\subsection{Related works}
In this sub-section, we briefly review previous related works for the composite problem \eqref{comp}. Then, we separately review theoretical and numerical results.  

\begin{itemize}
\item \textbf{Variational analysis of the optimization problem \eqref{comp}}
\end{itemize}
In all previous examples, either $g$ is non-convex, or $F$ is non-differentiable which itself may be in the form of compositions of many other non-differentiable functions. The latter irregularities cause a challenge to the study of variational analysis of the problem \eqref{comp} through the different notions of subdifferentials; see \cite{c, i, m18}. To the best of our knowledge, there is no result dealing with the composite optimization problems in such a general framework. The very recently published book \cite{cp} mostly concerns the composite functions $g \circ F$, where both $g$ and $F$ are locally Lipschitz continuous and directionally differentiable. The latter assumptions allow one to employ full limit instead of liminf/sup to define generalized directional derivatives. Consequently, the calculus rules are readily followed akin to the classical differentiable case. The price of such restrictions is to scarify the possibility of involving extended-real-valued functions. Due to this fact, authors in \cite{cp}, had to have a separate analysis of the computation of the tangent cones to constraint sets in various scenarios. Although from the view of exact penalization, Lipschitz continuous functions seem a sufficient ground to study first-order variational analysis of composite structures, the extended-real-valued functions are crucial for the study of second-order variational analysis of composite functions and constrained systems; for instance, the second subderivative of a Lipschitz continuous function may not be a finite-valued function; see \cite[Proposition~13.5]{rw}. The latter restrictions will also not allow one to have a unified analysis of amenable and non-amenable composite functions. Indeed, there is a comprehensive first and second-order order variational analysis on the composite problem \eqref{comp} where $g$ is convex and $F$ is continuously twice differentiable; see \cite{r88, r89, bp, rw, bs, ms, mms1, mms2}, we refer the reader to \cite{ash} for the latest and the strongest results in that topic.
\begin{itemize}
\item \textbf{Numerical methods for the optimization problem \eqref{comp}}
\end{itemize}
Most numerical methods concerning problem \eqref{comp} focus on the amenable cases, where $g$ is convex, and $F$ is smooth. The typical approach to treat amenable cases numerically is to linearize $F$ while keeping $g$ unchanged. The latter technique forces sub-problems for solving \eqref{comp} to become convex. Consequently, solving \eqref{comp} reduces to the solving of a sequence of convex problems. One typical case is the sum of a convex function and a smooth function. There is extensive literature around amenable cases; see e.g. \cite{bf, lw, zl, l, dp, ddmp, dil, r21}. Beyond amenability, the gradient sampling method may asymptotically find a Clarke stationary point for the non-smooth Lipschitz continuous functions \cite{blo, k} . Very recently in \cite{bl}, authors aimed to extend the convergence of the gradient sampling method to the directionally Lipschitz continuous functions. Quite recently, J. Zhang et. al \cite{zljsj} introduced a series of randomized first-order methods and analyzed their complexity in finding a $(\epsilon , \delta) -$ stationary point. The oracle of the algorithm in \cite{zljsj} needs to call a Clarke subgradient in each iteration. Both aforementioned methods are based on sampling (generalized) gradients, thus they do not reduce to the classical gradient descent when iterations happen on differentiable points. Very recently, the monograph \cite{cp} provides a systematic study of solving some classes of non-convex non-differentiable optimization problems by surrogate functions estimating the original function from above at each iteration. The subderivative method that we present in this paper can be viewed in the same line, minimizing an upper surrogate function for the optimization problem \eqref{comp}.     
\subsection{Preliminaries}
Throughout this paper, we mostly use the standard notations of variational analysis and generalized differentiation; see, e.g. \cite{m18,rw}. For a nonempty set $\Omega\subset X$, the notation $x\st{\O}{\to}\ox$ indicates that $x\to\ox$ with $x\in\O$, while $\co\O$ and $\cone\O$ stand for the convex and conic hulls of $\O$, respectively. The indicator function $\dd_\O$ of a set $\O$ is defined by $\dd_\O(x):=0$ for $x\in\O$ and $\dd_\O(x):=\infty$ otherwise, while the Euclidean distance between $x\in\X$ and $\O$ is denoted by $\dist(x;\O)$. Sometimes we need to use $\ell_1$-norm for the distance of two sets, in latter case, we again use the same term ''$\dist $", but we clarify the use of $\ell_1$ norm.  For a closed set $\O$, we  denote $\proj_{\O} (x)$, the Euclidean projection on the set $\O$. We write $x=o(t)$ with $x\in\X$ and $t\in\R_+$ indicating that $\frac{\|x\|}{t}\dn 0$ as $t\dn 0$, where $\R_+$ (resp.\ $\R_-$) means the collection of nonnegative (resp.\ nonpositive) real numbers. Recall also that $\N:=\{1,2,\ldots\}$.
Given a nonempty set $\O \subset \R^n $ with $\ox\in \O $, the  tangent cone $T_{\O}(\ox)$ to $\O$ at $\ox$ is defined by
\begin{equation*}\label{2.5}
T_{\O}(\ox)=\big\{w\in \R^n |\;\exists\,t_k{\dn}0,\;\;w_k\to w\;\;\mbox{ as }\;k\to\infty\;\;\mbox{with}\;\;\ox+t_kw_k\in  \O \big\}.
\end{equation*}
Given the extended-real-valued function $f:\R^n \to \oR:= (-\infty, \infty]$, its domain and epigraph are defined, respectively, by 
$$
\dom f =\big\{ x \in \R^n | \; f(x) < \infty \big \}\quad \mbox{and}\quad \epi f=\big \{(x,\al)\in X\times \R|\, f(x)\le \al\big \}.
$$
We say that $f \colon\R^n\to\oR$ is {\em piecewise linear-quadratic} (PLQ) if $\dom f =\cup_{i=1}^{s}\P_i$ with $\P_i$ being polyhedral convex sets for $i=1,\ldots,s$, and if $f$ has a representation of the form
\begin{equation}\label{PWLQ}
f(x)=\sm\langle A_ix,x\rangle+\langle a_i,x\rangle+\alpha_i\quad\mbox{for all}\quad x \in\O_i,
\end{equation}
where $A_i$ is an $n\times n$ symmetric matrix, $a_i\in\R^n$, and $\alpha_i\in\R$ for all $i=1,\ldots,s$. Further, the above function  $f$ is called {\em piecewise linear} (PL) if $A_i = 0$for all $i=1,\ldots,s$. Recall from \cite{mms1} that $f \colon\R^n\to\oR$ is {\em Lipschitz continuous} around $\ox\in\dom f$ {\em relative} to some set $\O\subset\dom f$ if there exist a constant $\ell\in\R_+$ and a neighborhood $U$ of $\ox$  such that
\begin{equation*}
| f(x)- f(u)|\le\ell\|x-u\|\quad\mbox{for all }\quad x,u\in\O\cap U.
\end{equation*}
 
The main advantage of defining Lipschitz continuity \textbf{relative} to a set, is to not missing the extended-real-valued functions for the study of constrained optimization. 
Piecewise linear-quadratic functions and indicator functions of nonempty sets are important examples of extended-real-valued functions that are Lipschitz continuous relative to their domains around any point $\ox\in\dom f$. Similarly we can define the calmness of $f$ \textbf{at} point $\ox$ relative to a set. In particular, $f$ is called calm at $\ox$ from below if there exist the $\ell > 0$ and a neighborhood $U$ of $\ox$ such that
\begin{equation*}
 f(x)- f(\ox) \geq  - \ell\|x- \ox\|\quad\mbox{for all }\quad x \in U.
\end{equation*}
If $f$ is Lipschitz continuous around $\ox$ relative to its domain, then $f$ is calm at $\ox$ from below. Given a function  $f:\R^n \to \oR$ and a point $\ox$ with $f(\ox)$ finite, the subderivative function $\d f(\ox)\colon\R^n\to[-\infty,\infty]$ is defined by \eqref{subderivatives}. Subderivative and tangent vectors are related by the relation $\epi \d f(\ox) = T_{\ss \epi f} (\ox , f(\ox))$. It is clear, from the latter tangent cone relationship, that $\d f(\ox)(.)$ is always lower semicontinuous, also $\d f(\ox)$ is proper if $f$ is calm at $\ox$ from below. Furthermore, We have $\dom \d f(\ox) = T_{\ss \dom f} (\ox )$ if $f$ is Lipschitz continuous around $\ox$ relative to $\dom f$, \cite[Lemma~4.2]{mm21}. If $f$ is convex or concave and $\ox \in \mbox{int} \; \dom f$ then \eqref{subderivatives} reduces to the usual directional derivative i.e,
\begin{equation}\label{dderivative}
{\mathrm d}f(\ox)(w)= f' (x ; w) : =\lim_{\substack{
   t\dn 0 \\
    }} {\frac{f(\ox+tw)-f(\ox)}{t}}.
\end{equation}
By a slight abuse of notation, we say that $f$ is directionally lower regular at $\ox \in \dom f$ if $f' (\ox ; w)$, as a number in $[- \infty , +\infty ]$, exists and it is equal to $\d f(\ox)(w)$. Beside the finite-valued-convex (concave) functions, there are many non-convex functions which are directionally lower regular, for instance, $f : = \delta_{X}$ where $X$ is a finite union of polyhedral convex sets,  $f$ is piecewise-linear quadratic, and $f(x)=\|x\|_0$, i.e., $\ell_0$-norm; see Example \ref{subzero}. Another example of such functions is the class of \textit{semi-differentiable} functions. Recall from \cite[Definition~7.20]{rw} that a vector-valued function $F:  \R^n \to \R^m$  is semi-differrentiable at $\ox$ for (direction) $w$, if the following limit exists 
\begin{equation}\label{semiderivatives}
{\mathrm d}F(\ox)(w):=\lim_{\substack{
   t\dn 0 \\
  w' \to w
  }} {\frac{F(\ox+tw')-F(\ox)}{t}}.
\end{equation}
$F$ is called semi-differentiable at $\ox$, if the above limit exists for all $w \in \R^n$, in this case, the continuous mapping $\d F(\ox) : \R^n \to \R^m$ is called the semi-derivative of $F$ at $\ox$. We say $F$ is semi-differentiable (without mentioning $\ox$) if $F$ is semi-differentiable everywhere in $\R^n$. As it is often clear from context, let us not use different notations for subderivative and semi-derivative; we use $\d$ for both them. Following \cite[Theorem~7.21]{rw} $F$ is semi-differentiable at $\ox$ if and only if there exists $h: \R^n \to \R^m$, continuous and homogeneous of degree one such that 
\begin{equation}\label{semidiff-ex}
F(x) = F(\ox) + h(x - \ox) + o (\| x - \ox \|),
\end{equation}
indeed  $h = \d F(\ox)$. It is easy to see that, for the locally Lipschitz continuous functions, the semi-differentiability and classical directional differentiability concepts are equivalent. In contrast with \cite[p.~156]{cp}, in the definition of semi-differentiability, we do not require $F$ to be locally Lipschitz continuous. However, one can observe that if $F$ is semi-differentiable at $\ox$, then $F$ is calm at $\ox$, that is there exists a $\ell > 0$ such that for all $x$ sufficiently close to $\ox$
$$   \| F(x)  -  F(\ox)\|  \leq \ell \; \| x -  \ox \|  .$$

\section{First-order variational analysis of non-amenable functions}\sce \label{sect02}
In this section, we study the first-order generalized directional derivative of the composite function $f: = g \circ F $ in problem \eqref{comp}. By developing a chain rule for subderivative, we calculate the subderivative of $f$. The obtained formula is applied to derive a tight first-order necessary optimality condition for problem \eqref{comp}. We first recall the definition of derivability from \cite[Definition~6.1]{rw} which plays an important role in our analysis of non-convex and non-Clarke regular constraint sets.
\begin{Definition}[\bf geometrically derivable sets] 
A tangent vector $w \in T_{X}(\ox)$ is said to be \textit{derivable}, if there exist $\varepsilon >0 $ and $\zeta : [0 , \varepsilon] \to X$ with $\zeta(0) = \ox$ and $\zeta'_{+} (0) = w$. The set $X$ is called geometrically derivable at $\ox$ if every tangent vector to $X$ at $\ox$ is derivable. We say $X$ is geometrically derivable if it is geometrically derivable at any $x \in X$.   
\end{Definition}
The class of geometrically derivable sets is considerably large. It is not difficult to check that the derivability is preserved under finite unions. Therefore, the finite union of convex sets is geometrically derivable as every convex set is geometrically derivable. A consequence of the latter result is that the non-convex polyhedron, graph, and epigraph of PL and PLQ functions are geometrically derivable. Furthermore, under a mild constraint qualification, the derivability is preserved under intersections and pre-image of smooth mappings \cite[Theorem~4.3]{mm21}. The derivability of $X$ at $\ox$ can equivalently be described by the relation $T_{X}(x) = \Lim_{t \dn 0} \frac{X - \ox}{t}$, where the limit means the set-limit in the sense of set-convergence \cite[p.~152]{rw}.\\
\begin{Definition}[\bf metric subregularity qualification condition] Let the composite function $f := g \circ F$ be finite at $\ox$, where $g : \R^m \to \oR$ and $F: \R^n \to \R^m  $. It is said that the metric subregularity qualification condition (MSQC) holds at $\ox$ for the composite function $f = g \circ F$ if there exist a $\kappa >0$ and $U$ an open neighborhood of $\ox$ such that for all $x \in U$ the following error bound holds :
\begin{equation}\label{msqc}
\dist ( x \; ; \; \dom f ) \leq \kappa \; \dist (F(x) \;  ; \; \dom g  ) 
\end{equation}
Similarly, we say the metric subregularity constraint qualification condition holds for the constraint set $\O  :=  \{ x \in \R^n | \;  G(x) \in X  \}$ at $\ox$ if MSQC holds for the composite function $\delta_{X} \circ G$, that is there exist a $\kappa > 0$ and $U$, an open neighborhood of $\ox$, such that for all $x \in U$ the following estimates holds:
\begin{equation}\label{msqccons}
\dist ( x \; ; \; \O ) \leq \kappa \; \dist (G(x) \;  ; \; X  )
\end{equation}
\end{Definition}
The above version of metric subregularity condition \eqref{msqc} for the composite function $g \circ F$ was first introduced in \cite[Definition~3.2]{mms1} by the author, which only concerns the domain of the functions $f$ and $g$, not their epigraphs as it was considered in \cite{io}. It is clear that the metric subregularity qualification condition is a robust property, in the sense that if \eqref{msqc} holds at $\ox$ then it also holds at any point sufficiently close to $\ox$. Also, observe that the metric subregularity qualification is invariant under changing of norms. The latter is due to the fact that norms are equivalent in finite dimensions. The metric subregularity condition \eqref{msqc} has shown itself a reasonable condition under which the first and second-order calculus rules hold for amenable setting both in finite and infinite-dimensional spaces; see \cite{mm21, mms1, ms, mms2, ash}. Metric subregularity constraint qualification is implied by Robinson constraint qualification. In particular, in nonlinear programming, it is implied by Mangasarian Fromovitz constraint qualification, and in convex programming, it is implied by Slater constraint qualification; for the comparison of various types of constraint qualifications see \cite[Proposition~3.1]{mms1} and \cite[Proposition~3.1]{mm21}. Next, we define a new metric subregularity qualification condition which is even weaker than \eqref{msqc}, and it turns out to be a more reasonable constraint qualification for the bilevel optimization problems; see Example \ref{bilevel}. 
\begin{Definition}[\bf directional metric subregularity]\label{dirmscq} Let the composite function $f := g \circ F$ be finite at $\ox$. Take $w \in \R^n$. We say the directional metric subregularity qualification condition holds at $\ox$ in direction $w$ if \eqref{msqc} holds restricted to the half-line $\{ \ox + tw | \; t\geq 0\}$, that is there exist a $\kappa = \kappa (\ox ,w) > 0$ and a $\epsilon=\varepsilon(\ox , w) >0$ such that for all $t \in [0 ,\varepsilon]$

\begin{equation}\label{dmsqc}
\dist ( \ox + t w \; ; \; \dom f ) \leq \kappa \; \dist (F(\ox + t w) \;  ; \; \dom g  ) 
\end{equation}
We say the directional metric subregularity condition holds $\ox$ (without specifying any direction) if \eqref{dmsqc} holds for all  for each $w \in \R^n$. Similarly, we can define the directional metric subregularity constraint qualification for the constraint set $\O  :=  \{ x \in \R^n | \;  G(x) \in X  \}$, i.e.,
\begin{equation}\label{dmsqccons}
\dist ( \ox+tw \; ; \; \O ) \leq \kappa \; \dist (G(\ox+tw) \;  ; \; X  )
\end{equation}
\end{Definition} 
The directional metric subregularity (along all directions) is weaker than (full) metric subregularity \eqref{msqc}, thus any sufficient condition for metric subregularity serves a sufficient condition for directional metric subregularity. If $\kappa $ and $\epsilon$ in \eqref{dmsqc} can be taken uniformly in $w$ then both metric subregularity and directional metric subregularity coincide. Besides sufficient conditions, the checking of \eqref{dmsqc} directly along all $w$ is easier than checking \eqref{msqc}, due to simplicity in dimension one. In  Definition \ref{dirmscq}, one can restrict direction $w \in \R^n$,  to $w$ with $\|w\|=1$.  If $w$ is feasible direction at point $\ox \in \O$, i.e., $\ox + t w \in \O$ for sufficiently small but positive $t$, then the directional metric subregularity \eqref{dmsqccons} holds with any constant $\kappa > 0 .$ Therefore, to checking \eqref{dmsqccons}, we only need to establish \eqref{dmsqccons} for infeasible directions. The directional metric subregularity \eqref{dmsqc} is also weaker than the directional metric subregularity introduced by Gfrerer \cite[Definition~2.1]{g} , which requires the estimate \eqref{msqc} holds for all $x$ belong to a \textbf{set} having the set $\{ \ox + t w | t \in ( 0 , \varepsilon  ]  \}$ within its interior. Therefore, all sufficient conditions for metric subregularity \eqref{msqc}, and the conditions in \cite[Theorem~4.3]{g} and \cite[Theorem~4.1]{byz} can serve as sufficient conditions for \eqref{dmsqc} and \eqref{dmsqccons}.  We prefer to keep the term \textit{directional} in definition \eqref{dirmscq}, since \eqref{dmsqc} is metric subregularity merely along direction $w$, not relative to a set which was considered in \cite[Definition~2.1]{g}.  
\subsection{Calculus rules via subderivative}

In the following theorem, we establish a subderivative chain rule for the composite function $g \circ F$ where $F$ is semi-differentiable. A non-direction version of the latter chain rule previously obtained by author in \cite{mms1} and \cite{ash}  under stronger assumptions; for example, in \cite[Theorem~3.4]{mms1}, $F$ was assumed to be differentiable, and the (full) metric subregularity \eqref{msqc} holds. The latter would allow us to replace $F$ by its first-order taylor expansion to reduce the problem to the problem of the chain rule for the composite function $g \circ A$, where $A$ is an linear map. However, a careful modification of the proof in \cite[Theorem~3.4]{mms1} can be still used here. We give a detailed proof for the sake of completeness.   
\begin{Theorem}[\bf subderivative chain rule of non-amenable functions]\label{fcalc}
Let $f: = g \circ F $ be finite at $\ox$ where $g: \R^m \to \oR$ is Lipschitz continuous relative to $\dom g$, and $F:\R^n \to \R^m $ is semi-differentiable at $\ox $ in direction $w$. Further, assume that the directional MSCQ \eqref{dmsqc} holds at $\ox$ for direction $w$. Then, the following chain rule holds:
\begin{equation}\label{fcalc1}
\d f(\ox) (w) = \d g ( F(\ox)) ( \d F(\ox) w)  
\end{equation}
\end{Theorem}
\begin{proof} Define $\eta(w' ,t) := F(\ox+tw')-F(\ox)- t\d F(\ox)(w)$. Deduce from the semi-differentiability of $F$ at $\ox$ for $w$ that $\frac{\eta(w' ,t)}{t}\to 0$ as $t\dn 0$ and $w' \to w$. By using the function $\eta$ and definition \eqref{subderivatives}, we get the relationships
\begin{eqnarray*}
\disp\d(g \circ F)(\ox)(w)&=&\liminf_{\substack{t\dn 0\\w'\to w}}\frac{g\big(F(\ox+t w')\big)-g \big(F(\ox)\big)}{t}\nonumber\\
&=&\liminf_{\substack{t\dn 0\\w'\to w}}\frac{g\big(F(\ox)+t\d F(\ox)w + \eta(w' , t) \big)-g \big(F(\ox)\big)}{t}\nonumber\\
&=&\liminf_{\substack{t\dn 0\\w' \to w}}\frac{g\big(F(\ox)+t(\d  F(\ox)w+\frac{\eta(w',t))}{t})\big)-g \big(F(\ox)\big)}{t}\nonumber\\
&\ge&\d g \big(F(\ox)\big)\big(\d F(\ox) w \big)\;\mbox{ whenever }\; w\in\R^n,
\disp
\end{eqnarray*}
which verify the inequality ``$\ge$" in \eqref{fcalc1}. Proceeding next with the proof of the opposite inequality in \eqref{fcalc1}, if $\d g (F(\ox))(\d F(\ox)w)=\infty$, the latter inequality holds obviously. Also observe from the Lipschitz continuity of $g$ around $F(\ox)$ relative to its domain that $\d g (F(\ox))(\d F(\ox)w)>-\infty$. Therefore, we may assume that the value $\d g(F(\ox))(\d F(\ox)w)$ is finite, and thus there exist sequences $t_k\dn 0$ and $v_k \to \d F(\ox)w$ such that
\begin{equation}\label{subd0}
\d g \big(F(\ox)\big)\big(\d F(\ox)w\big)=\lim_{\substack{k\to\infty}}\frac{ g \big(F(\ox)+t_k v_k\big)- g \big(F(\ox)\big)}{t_k} \in \R.
\end{equation}
Suppose without loss of generality that $F(\ox)+t_k v_k\in\dom g$ for all $k\in\N$. Then, the imposed directional MSQC \eqref{dmsqc} at $\ox$ in direction $w$ yields
\begin{equation*}
{\rm dist}(\ox+t_kw \; ; \; \O)\le\kappa\,{\rm dist}\big(F(\ox+t_k w)\; ;\; \dom g \big),\quad k\in\N,
\end{equation*}
which in turn brings us to the relationships
\begin{eqnarray*}
{\rm dist}\Big(w \; ; \; \frac{\O-\ox}{t_k}\Big)&\le &\frac{\kappa}{t_k}\,{\rm dist}\big(F(\ox)+t_k\d F(\ox)w+ \eta(t_k , w) \; ; \; \dom g \big)\nonumber \\
&\le&\frac{\kappa}{t_k}\,\big\| F(\ox)+t_k\d F(\ox)w+\eta(t_k , w) -F(\ox)-t_k v_k\big\|\nonumber\\
&=&\kappa\,\Big\|\d F(\ox)w-v_k+\frac{\eta(t_k , w) }{t_k}\Big\|\;\mbox{ for all }\;k\in\N.
\disp
\end{eqnarray*}
It allows us to find vectors $w_k\in\frac{\O-\ox}{t_k}$ satisfying
$$
\|w-w_k\|\le \kappa\,\Big\|\d F(\ox)w-v_k+\frac{\eta(t_k , w) }{t_k}\Big\| + \frac{1}{k}
$$
and hence, telling us that $\ox+t_k w_k\in\O$ for all $k\in\N$ and that $w_k\to w$ as $k\to\infty$. Since $g$ is Lipschitz continuous relative to its domain around $F(\ox)$, there exists $\ell >0$ such that for $k \in \N$ sufficiently large we have
$$ g\big(F(\ox)+t_k v_k)\big)- g\big(F(\ox+t_k w_k)) \geq - \ell \; \| F(\ox+t_k w_k)-F(\ox)- t_k v_k \| .$$ 
Combining above with \eqref{subd0}, we arrive at the relationships 
\begin{eqnarray*}
\disp \d g\big(F(\ox)\big)\big(\d F(\ox)w\big)&=&\disp\lim_{k\to\infty}\Big[\frac{ g\big(F(\ox+t_k w_k)\big)- g \big(F(\ox)\big)}{t_k}+\frac{g\big(F(\ox)+t_k v_k)\big)- g\big(F(\ox+t_k w_k)\big)}{t_k}\Big]\\
&\ge& \disp\liminf_{k\to\infty}\frac{g \big(F(\ox+t_kw_k)\big)-g\big(F(\ox)\big)}{t_k}-\ell\limsup_{k\to\infty}\Big\|\frac{F(\ox+t_k w_k)-F(\ox)}{t_k}-v_k\Big\|\\
&\ge&\d(g\circ F)(\ox)(w)-\ell\limsup_{k\to\infty}\Big\| \d F(\ox)w +\frac{\eta(w_k , t_k)}{t_k}-v_k\Big\|=\d (g\circ F)(\ox)(w),\disp
\end{eqnarray*}
This verifies the inequality ``$\ge$" in \eqref{fcalc1} and completes the proof of the theorem.
\end{proof}
\begin{Example}[\bf subderivative of zero norm]\label{subzero}{\rm Let $A \in \R^{m \times n}$ be a matrix and $b \in \R^m$ be a column vector. For each $x \in \R^n,$ define $f(x) := \| Ax + b\|_{0}$, where $\|y\|_{0}$ counts the number of nonzero components of $y$. Then, $f$ is directionally lower regular at each $x \in \R^n$, and the subderivative of $f$ is calculated by
\begin{eqnarray*}
\disp
\d f(x)(w)  = \left\{\begin{matrix}
0 & \;S(Aw) \subseteq S(Ax + b) \\ 
 \infty & \; S(Aw) \nsubseteq S(Ax + b)
\end{matrix}\right. 
\disp
\end{eqnarray*}
where $S(y)$ denotes the support of $y \in \R^m$, that is the set of all $i \in \{1,2,...,n\}$ with $y_i \neq 0  $.
}
\end{Example}
\begin{proof}
Take $x, \;w \in \R^n$, and $W \subset \R^n$, a bounded neighborhood of $w$. It is not difficult to show that there exists a $\varepsilon >0$ such that for all $t \in (0 , \varepsilon)$ and all $w' \in W$ one has
$$ f(x+tw')= \| Ax +b +t Aw'  \|_{0} = \|Ax + b \|_{0}  +  \mbox{Card}\; \big( S(Aw') \setminus S(Ax + b)  \big), $$
where ''Card(A)" stands for the number of elements in the set $A$. It is not difficult to check that $S(Aw) \subseteq S(Aw')$, for all $w'$ sufficiently close to $w$. The latter implies that
$$f(x+tw) \leq f(x+tw') \quad \mbox{for all}~ t\in (0,\varepsilon) $$
Therefore, for all $t \in \in (0,\varepsilon)$ and $w'$ close to $w$ we have 
\begin{equation*}
\frac{f(x+tw) - f(x)}{t} \leq  \frac{f(x+tw') - f(x)}{t} = \frac{\mbox{Card}\; \big( S(Aw') \setminus S(Ax + b)  \big)}{t} 
\end{equation*}
Hence,
\begin{eqnarray*}
\disp
\d f(x)(w) &=& \liminf_{t \dn 0} \; \frac{f(x+tw) - f(x)}{t} \\\nonumber 
&=& \lim_{t \dn 0} \; \frac{\mbox{Card}\; \big( S(Aw) \setminus S(Ax + b)  \big)}{t} = \left\{\begin{matrix}
0 & \;S(Aw) \subseteq S(Ax + b) \\ 
 \infty & \; S(Aw) \nsubseteq S(Ax + b)
\end{matrix}\right. 
\disp
\end{eqnarray*}
The above relationships show that $f$ is directionally lower regular. 
\end{proof}
In Theorem \ref{fcalc}, if further  $g$ is semi-differentiable at $F(\ox)$ for the direction $\d F(\ox)(w)$,  then  $g \circ F$ is semi-differentiable at $\ox$ for the direction $w$. In particular, if both $F$ and $g$ are semi-differentiable then so is $g \circ F$ and thus it is directionally lower regular. Note that in the latter case, the proof can get significantly simpler as one deal with (full) \textit{limit} instead of \textit{liminf }, similar to the clasical differentiable chain rule. Moreover, in the semi-differentiable case, due to the existence of the full limit, we can take $g$ a vector-valued function and drop the relative Lipschitz continuity assumption on $g$. For the sake of completeness, we give a short proof for the semi-differentiable chain rule, as it seems to be absent in the literature.       
\begin{Proposition}[\bf chain rule for semi-differentiable functions]\label{semidiff-chain}
Let $F: \R^n \to \R^m$ be semi-differentiable at $\ox$ for direction $w$ and $G: \R^m \to \R^p$ be semi-differentiable at $F(\ox)$ for the direction $\d F(\ox)(w)$ then $G \circ F$ is semi-differentiable at $\ox$ for the direction $w$ with
\begin{equation}\label{semidiff-chain1}
\d (G \circ F)(\ox) (w) = \d G(F(\ox)) (\d F(\ox)w).
\end{equation}
In particular, $G \circ F$ is semi-differentiable at $\ox$ if $F$ and $G$ are semi-differentiable at $\ox$ and $F(\ox)$ respectively. 
\end{Proposition}
\begin{proof}
Define $\eta(w' ,t) := F(\ox+tw')-F(\ox)- t\d F(\ox)(w)$. Deduce from the semi-differentiability of $F$ at $\ox$ for $w$ that $\frac{\eta(w' ,t)}{t}\to 0$ as $t\dn 0$ and $w' \to w$.  Following relationships hold
\begin{eqnarray*}
	\disp\d(G \circ F)(\ox)(w)&=&\lim_{\substack{t\dn 0\\w'\to w}}\frac{G\big(F(\ox+t w')\big)-G \big(F(\ox)\big)}{t}\nonumber\\
	&=&\lim_{\substack{t\dn 0\\w'\to w}}\frac{G\big(F(\ox)+t\d F(\ox)w + \eta(w' , t) \big)-G \big(F(\ox)\big)}{t}\nonumber\\
	&=&\lim_{\substack{t\dn 0\\w' \to w}}\frac{G\big(F(\ox)+t(\d  F(\ox)w+\frac{\eta(w',t))}{t})\big)-G\big(F(\ox)\big)}{t}\nonumber\\
	&=&\d G \big(F(\ox)\big)\big(\d F(\ox) w \big),
	\disp
\end{eqnarray*}
where the last equality holds due to the semi-differentiability of $G$ at $F(\ox)$ for the direction $\d F(\ox)(w).$
\end{proof}
In contrast with subdifferential chain rules, both chain rules \eqref{fcalc1} and \eqref{semidiff-chain1} hold in the form of equality. These major pros allow us to apply these chain rules over the arbitrary number of functions composed of each other. The latter is particularly useful to compute the subderivative of the loss function of a deep neural network in \eqref{l2los}. In this regard, we have the following corollary for which we omit its proof. 
\begin{Corollary}[\bf auto-differentiation via semi-derivative]
Let $F_i$, $i=1,...k$ be  semi-differentiable functions on $\R^n$. Let $x , w \in \R^n$.  The the semi-derivative of the composition $F_k \circ F_{k-1} \circ ... F_{1}$ can be calculated through the following recursive way:
\begin{equation*}
\left\{\begin{matrix}
x_0 := x  \\ 
 x_i :=  F_{i}(x_{i-1})&  i=1,...,k \\
  u_i (.) :=  \d F_{i}(x_{i-1})(.)&  i=1,...,k 
\end{matrix}\right.
\end{equation*}
The semi-derivative of the above composite function  is $u_k$, i.e.,
$$  \d \big( F_k \circ F_{k-1} \circ ... F_{1} \big)(x)(.) =  u_k(.) $$
\end{Corollary}
\begin{Proposition}[\bf semi-differentiability of the distance functions]\label{semidist}{\rm Let $X$ be a nonempty closed set. Set $f(x):= \dist (x ; X)$. Then, the subderivative of $f$ is computed by
\begin{equation}\label{subdist}
\d f(x)(w) =\left\{\begin{matrix}
\dist (w \;  , \;  T_{X} (x)) & x \in X   \\\\  
   \min \{ \frac{\la x - y \;  , \; w\ra}{f(x)} \big|  \;y \in \proj_{X}(x) \}&   x \notin X 
\end{matrix}\right.
\end{equation}
Furthermore, $f$ is semi-differentiable if $X$ is geometrically derivable. 
}
\end{Proposition}
\begin{proof}
The expression for  the subderivative of $f$ in \eqref{subdist} was obtained in \cite[Example~8.53]{rw}. We only give a proof to the second part. Fix $\ox , w \in \R^n$. Since $f$ is (globally) Lipschitz continuous, it suffices to show that
$$ \d f(\ox)(w) = \limsup_{t \dn 0} \frac{f(\ox + t w) - f(\ox)}{t}.  $$
According to the definition of subderivative we always have $\leq$. To prove the opposite inequality, first,  we assume $\ox \in X$, thus $f(\ox) = 0$. Since $X$ is geometrically derivable, we have $$ T_{X}(\ox) =\Lim_{t \dn 0}\frac{X - \ox}{t}.$$ By the connection of limits and distance functions in \cite[Corollary~4.7]{rw},  we can write 
\begin{eqnarray*}
\disp
\d f(\ox)(w) &=& \dist (w \;  ; \;  T_{X} (\ox)) =  \dist \big( w \; ; \; \Lim_{t \dn 0}\frac{X - \ox}{t} \big) \\\nonumber
&=& \lim_{t \dn 0} \dist \big( w \; ; \; \frac{X - \ox}{t} \big) = \lim_{t \dn 0} \frac{f(x + t w) - f(x)}{t}.
\disp
\end{eqnarray*}
Now we suppose $\ox \notin X$. By \eqref{subdist}, we can find $\oy \in \proj_{X} (\ox)$ such that $\d f(\ox)(w) =  \frac{\la \ox - \oy \;  , \; w\ra}{f(\ox)} $. Define $\eta (t):= \|\ox + tw - \oy \|$. The function $\eta$ is differentiable at $t = 0$, because $\ox \neq \oy$. Since $f(\ox + tw) \leq \| \ox + tw - \oy \|$ for all $t \geq 0$, we have 
$$ \frac{f(\ox + t w) - f(\ox)}{t} \leq \frac{\eta(t) - \eta(0)}{t} \quad \quad t>0 $$ 
Hence
$$ \limsup_{t \dn 0} \frac{f(\ox + t w) - f(\ox)}{t} \leq \eta' (0) = \frac{\la \ox - \oy \;  , \; w\ra}{\| \ox -\oy \|} = \d f(\ox) (w) $$
which proves the semi-differentiability of $f$ at $\ox$. 
\end{proof}
\begin{Remark}{\rm
		In the light of Theorem \ref{fcalc} and Proposition \ref{semidist} the  the penalized function associated with the constraint $G(x) \in X$ is semi-differentiable if $X$ is geometrically derivable and $G$ is semi-differentiable. More precisely, in the constrained optimization \eqref{ccomp}, if $\varphi := g \circ F$ and $G$ are semi-differentiable at $x$, and $X$ is geometrically derivable at $G(x)$, then the function $ f := \varphi + \rho \; \dist (G(.) ; X)$ is semi-differentiable, consequently, it is directionally lower regular at $x$. These conditions are general enough to cover interesting classes of non-amenable functions, such as examples mentioned in section \ref{sect01}. We will see that directional lower regularity, in particular, semi-differentiability of $f$, is a key to establish first-order descent methods for the penalized problems.}
\end{Remark}
An immediate consequence of a chain rule is sum rule that can be obtained by Theorem \ref{fcalc}. A non-directional version of this sum rule had been previously obtained in \cite[Theorem~4.4]{mm21}. 
\begin{Corollary}[\bf sum rule via subderivative]\label{fsum}
Let functions $f,  h : \R^n \to \oR$ be finite at $\ox$. Let $f$ and $h$ be Lipschitz continuous around $\ox$ relative to their domains. Furthermore, assume that the metric subregularity condition holds at $\ox$ in direction $w$ for $\dom f \cap \dom g$ that is, there exist a $\kappa = \kappa(\ox , w) > 0 $ and a $\varepsilon = \varepsilon (\ox , w) >0$ such that for all $t \in [0 , \epsilon]$, the following estimate holds: 
\begin{equation}\label{summsqc}
\dist  (\ox + tw  \; ; \; \dom f \cap \dom h ) \leq \kappa \; \big( \dist  (\ox + tw \ ; \dom f)  +  \dist (\ox +tw \; ; \; \dom  h) \big).
\end{equation}
Then, the subderivative sum rule holds at $\ox$ for the direction $w$:	 
\begin{equation}\label{fsum1}
\d (f+h)(\ox)(w) =  \d f(\ox)(w) + \d h(\ox)(w) 
\end{equation}

\end{Corollary}
\begin{proof}
Define $g: \R^n \times \R^n \to \oR$ with $g(x ,y):= f(x) + h(y)$ and  $F: \R^n \to \R^n \times \R^n$ with $F (x) = (x ,x)$. It is easy to check that the estimate \eqref{summsqc} reduces to the directional metric subregularity \eqref{dmsqc} for the composite function $g \circ F$ at $\ox$ for direction $w$. By applying chain rule \eqref{fcalc1} on previous composite function, we arrive at \eqref{fsum1}.
\end{proof}
Note that if one of the functions $f$ or $h$ be finite-valued, then \eqref{summsqc} holds automatically for all direction $w \in \R^n$. In particular, if one the functions $f$ or $h$ is semi-differentiable at $\ox$,  then the sum rule \eqref{fsum1} holds for all direction $w \in \R^n$ without assuming relative Lipschitz continuity. The both chain rule \eqref{semidiff-chain1} and sum rule \eqref{fsum1} will lead us to the calculation of tangent vectors of constraint sets,  which is a key to (primal) first-order necessary optimality condition of the problem \eqref{ccomp}.
\begin{Corollary}[\bf calculation of tangent vectors to the constraint sets]\label{calctan} Let $\ox \in \O:= \{x \big|\; G(x) \in X \}$ where $G: \R^n \to \R^m$ is semi-differentiable at $\ox$ for all directions $w \in C:= \{ w \big|\;   \d G(\ox)(w) \in T_{X} (G(\ox))  \}$, and the directional metric subregularity constraint qualification \eqref{dmsqccons} holds at $\ox$ for all such directions, then
\begin{equation}\label{calctan1}
T_{\O}(\ox) = \{ w \in \R^n \big| \; \d G(\ox)(w) \in T_{X} (G(\ox))   \}
\end{equation} 

\end{Corollary}
\begin{proof}
Observe that we always have $T_{\O}(\ox) \subseteq \{ w \in \R^n \big| \; \d G(\ox)(w) \in T_{X} (G(\ox))   \} = C$ without assuming any constraint qualification. To show the opposite inclusion, pick $ w \in C$. By chain rule \eqref{fcalc1}, we have $$\d \delta_{\O}(\ox)(w) = \d \delta_{C} ( G(\ox)) (\d G(\ox) w) = 0$$ which forces $w \in \dom \d \delta_{\O}(\ox) = T_{\O}(\ox).$ 
\end{proof}

\begin{Proposition}[\bf first-order necessary optimality condition]\label{fonc}
Let $\ox$ be  a local optimal solution for problem \eqref{ccomp}, where $X$ is a closed set and $g$ is a locally Lipschitz continuous function. Let $C \subseteq \R^n$ be a nonempty set. Assume that $F$ and $G$ are  semi-differentiable at $\ox$ respectively for all $w \in C$. Finally suppose that the directional metric subregularity constraint qualification \eqref{dmsqccons} holds at $\ox$ for all $ w \in C$. Then, 
\begin{equation}\label{fonc1}
 \d g ( F(\ox)) ( \d F(\ox)w ) \geq 0 \quad \quad \text{for all} ~ w \in C ~ \mbox{with}~ \d G(\ox) w \in T_{X}(G(\ox)).
\end{equation}
\end{Proposition}
\begin{proof}
Let $f : \R^n \to \oR$ be the essential function of the optimization problem \eqref{ccomp},  i.e., $f = g \circ F + \delta_{X} \circ G$. Since $\ox$ is a local minimizer of the problem $\min_{x \in \R^n} ~ f(x) ,$ we have $\d f (\ox)(w) \geq 0$ for all $w \in \R^n . $ On the other hand, by the sum rule \eqref{fsum1} and the chain rule \eqref{fcalc1}, for all $w \in C$ we have 
\begin{eqnarray*}
\disp
\d f(\ox) (w) &=& \d (g \circ F) (\ox)(w)  + \d (\delta_{X} \circ  G)(\ox)(w) \\\nonumber
&=& \d g (F(\ox)) (\d F(\ox)w) + \d \delta_X (G(\ox))(\d G(\ox)w)\\\nonumber
&=& \d g (F(\ox)) (\d F(\ox)w) + \delta_{\ss T_{X}(G(\ox))} (\d G(\ox)w )
\disp
\end{eqnarray*} 
This verifies \eqref{fcalc1}.
\end{proof}
The bigger set $C$ is, the sharper the necessary condition \eqref{fonc1} is. Particularly, we are interested in the case $C = \R^n .$ Since we alwas have $T_{\O} \subseteq \{w| \; \d G(\ox)(w) \in T_{X}(G(\ox))\}$, the set $C$ does not need to be larger than $\{w| \; \d G(\ox)(w) \in T_{X}(G(\ox))\}$. Hence, one does need to worry about semi-differentiability or directional metric subregularity in all directions, as any differentiability property constraint qualification comes with a cost in the bilevel programming; see Example \ref{obilevel}. Motivated by the latter observation, we have the following definition of stationarity relative to a given set. 
\begin{Definition}[\bf directional stationary points relative to sets]\label{ds} We say that a feasible point $\ox$ for problem \eqref{ccomp} is  a \textit{(d)irectional-stationary} point relative to the set $C$, if it satisfies the first-order necessary condition \eqref{fonc1}. The feasible point $\ox$ is called a d-stationary point (without mentioning any set), if it is a d-stationary point relative to the set $C:=\{w| \; \d G(\ox)(w) \in T_{X}(G(\ox))\}$.
\end{Definition}
It is well-known that the d-stationary points are the closest candidates to the local minimizers among other stationary points, this can be illustrated in a ReLU neural network function; indeed, for the function $f(x) =  - \max \{0 , x \}$, the point $\ox =0$ is identified as a Clarke- even a  limiting-stationary point, whereas, it is not a d-stationary point which correctly rejects its optimal candidacy. To compare the different types of stationary points we refer the reader to \cite[Proposition~1]{cps}. When an optimization problem lacks a constriant qualification or direffrentiability on its datas, e.g. bilevel programming  and opimization with complimentarity constraints,  the d-stationarity relative to smaller sets can be used. To illustrait the latter situation, in the following example and remark, we compute the d-stationary points of the bilevel programming formulated in Example \ref{bilevel}.    

\begin{Example}[\bf optimality condition of bilevel programming]\label{obilevel}{\rm Let $(\ox , \oy) \in \R^n \times \R^m$, and let $D \subset \R^n $ be a nonempty set. In the setting of the Example \ref{bilevel}, assume that the value function $V$ is semi-differentiable for all directions $w \in D .$ If $(\ox , \oy)$ solves the problem (BP), it solves (VP) as well. Furthermore, we assume that in (VP) the directional metric subregularity CQ \eqref{dmsqccons} holds at $\ox$ for the all directions $(w , u) \in D \times \R^m $. Theforefore, by Proposition \ref{fonc}, we have $(\ox , \oy)$ is a $d$-stationary point for (VP) relative to the set $D \times \R^m$ that is 
$$ \nabla \varphi (\ox , \oy) (w , u) \geq 0  $$
whenever $(w , u) \in D \times \R^m$ and
\begin{equation}\label{obilevel1}
\la  \nabla \psi (\ox , \oy) \; , \;  (w , u) \ra  - V' (\ox ; w) \leq 0, \quad \la \nabla h(\ox , \oy) \; , \; (w,u) \ra \leq 0
\end{equation}
In the above, we implicitly assume $h(\ox , \oy) =0$, otherwise we could neglect the constraint $h(x , y) \leq 0 $, thus drop the related inequality in \eqref{obilevel1}. For the calculation of $V' (\ox ; w)$ in terms of the initial dates $h$ and $\psi$, one can see e.g., \cite[Proposition~4.3]{by} or \cite[Theorem~30]{mins}. 
}
\end{Example}
\begin{Remark}[\bf comparison with other results]{\rm In Example \ref{obilevel}, we only require $V$ to be semi-differentiable at $\ox$ for directions $w \in D$. The latter is implied by directional differentiability and \textit{directional Lipschitz continuity}, defined in \cite{by}, of the value function $V$; see \cite[Theorem~4.1]{by} and \cite[Proposition~4.3]{by}. The optimality condition obtained in Example \ref{obilevel} is in a primal version which is generally sharper than a dual version obtained by a subdifferential. In this paper, we only stick to the primal argument. However, one can obtain a dual optimality condition for the constrained problem \eqref{ccomp} and Example \ref{obilevel} by adapting a suitable directional subdifferential. Alternatively, under the assumption of the existence of critical direction made in \cite[Theorem~5.1]{by}, we can derive a dual-directional optimality condition. Indeed, in Example \ref{obilevel}, if $w \in D $ and there exists $v \in \R^m$ such that $(w , v)$ satisfying \eqref{obilevel1} with $ \nabla \varphi (\ox , \oy) (w , v) \leq 0 $, then our directional optimality condition tells us that the following linear programming (in variable $u$) has the optimal solution $v$ with the optimal value $- \la \nabla_{x} \varphi (\ox , \oy) , w \ra$:

\begin{equation}\label{lp}
\mbox{minimize} \;\;   \la \nabla_{y} \varphi (\ox , \oy) , u \ra\quad \mbox{subject to} \quad  u ~ \mbox{satisfies} ~ \eqref{obilevel1}
\end{equation}
Now one by taking a dual of the above linear programming will arrive to a dual-version of the optimality condition for the bilevel problem \eqref{bilevel}. Indeed the dual of problem \eqref{lp} is
\begin{eqnarray}\label{dlp}
\disp
&\mbox{maximize}& - \lambda_1 V'(\ox ; w) + \lambda_1 \la \nabla_{x} \psi (\ox  , \oy) , w  \ra + \lambda_2 \la \nabla_{x} h (\ox  , \oy) , w  \ra \\\nonumber
&\mbox{subject to}& \lambda_1 \nabla_{y} \psi (\ox  , \oy) + \lambda_2 \nabla_{y} h (\ox  , \oy)= - \nabla_{y} \varphi(\ox , \oy) ,\\\nonumber
&& \lambda_1  \geq 0 , \lambda_2 \geq 0
\end{eqnarray}
Thanks to the strong duality argument in linear programming, problems \eqref{lp} and \eqref{dlp} have the same optimal value which is $- \la \nabla_{x} \varphi (\ox , \oy) , w \ra$. Therefore, the dual necessary optimality condition for the bilevel problem \eqref{bilevel} in direction $w \in D$ amounts to existence of $\lambda_1 , \lambda_2 \in \R$ such that
\begin{equation}\label{sipeqop1}
\left\{\begin{matrix}
\disp  \lambda_1 V'(\ox ; w) - \lambda_1 \la \nabla_{x} \psi (\ox  , \oy) , w  \ra - \lambda_2 \la \nabla_{x} h (\ox  , \oy) , w  \ra =  \la \nabla_{x} \varphi (\ox , \oy) , w \ra ,\\\\
\lambda_1 \nabla_{y} \psi (\ox  , \oy) + \lambda_2 \nabla_{y} h (\ox  , \oy)= - \nabla_{y} \varphi(\ox , \oy),\\\\
\lambda_1  \geq 0 , \lambda_2 \geq 0.
\end{matrix}\right.
\end{equation}
However, in this paper,  we are not dealing with dual optimality conditions as we believe that in non-amenable structures dual elements might not carry helpful information. To the best of our knowledge the primal optimality condition obtained in Example \ref{obilevel}, firstly, is the sharpest optimality condition for a general bilevel problem, secondly, it requires minimal assumptions compared to the other similar results obtained in the literature, in particular, the ones in the very recent paper \cite[Theorems~ 5.2, 5,3]{by}.   
}
\end{Remark}

\section{Subderivative Method}\sce \label{sec03}
In the sequel, we focus on the following unconstrained optimization problem 
\begin{equation}\label{uncons}
\mbox{minimize} \;\;  f(x) \quad \mbox{subject to}\;\; x \in \R^n ,
\end{equation}
where $f: \R^n \to \oR $ is calm at all $x \in \R^n$ from below. We are going to propose a first-order algorithm to obtain a d-stationary point for problem \eqref{uncons} up to some positive tolerance $\varepsilon > 0$.  To make the latter precise,  fix $\varepsilon > 0 $. We say $\ox$ is a $\varepsilon$-d(irectional) stationary point of problem \eqref{uncons} if  $\d f(\ox)(w) \geq - \epsilon$ for all $w \in \R^n$ with $\|w\|  \leq 1$.  Since we do not require $f$ takes only finite values, we can incorporate the constraint set into objective funtion of \eqref{uncons} with indicator function. However, from computation purpose, it is better we  incorporate constraint sets  into objective function by a finite exact penalty. More precisly, in the constrained optimization problem  
\begin{equation}\label{conso}
	\mbox{minimize} \;\;  \varphi(x) \quad \mbox{subject to}\;\;  G(x) \in X ,
\end{equation}
the constraint set can be incorprated into objective function by the penalty term $\rho \; \dist (G(x) \; ; \; X)$ where $\rho > 0$ is the penalty constant. Therefore, setting $f(x) := \varphi(x) + \rho \; \dist (G(x) \; ; \; X)$ in \eqref{uncons} we can cover constrained optimization problems.       
\vspace{.5cm}
 \begin{table}[H]
        \renewcommand{\arraystretch}{1}
        \centering
        \label{alg1}
        \begin{tabular}{|
                p{0.9\textwidth} |}
            \hline
            \textbf{Subderivative Method}\\
 
            \begin{enumerate}[label=(\alph*)]
            
                \item[0.] \textbf{(Initialization)} Pick the tolerance $\epsilon \geq 0$, starting point $x_0 \in \R^n$, and set $k=0 .$
           
                \item[1.] \textbf{(Termination)} If $\min_{\|w\| \leq 1}  \d f(x_k)(w) \geq  - \varepsilon $ then Stop.
                \item[2.] \textbf{(Direction Search)} Pick $w_k \in \mbox{arg min}_{\|w\| \leq 1 }   \d f(x_k)(w).$
                \item[3.] \textbf{(Line Search)} Choose the step size $\alpha_{k}  > 0$ through a line search method.
                \item[4.] \textbf{(Update)} Set $x_{k+1} : = x_k + \alpha_k  w_k$ and $k+1  \leftarrow k$ then go to step 1. 
            \end{enumerate}
            \\ \hline
        \end{tabular}
    \end{table}
\vspace{.3cm}

\begin{Remark}[\bf comments on the subderivative method]{\rm

\begin{itemize}
\item[]
\item[\textbullet] The subderivative method does not require any more assumptions on $f$ in order to be well-defined. Indeed, for any $f: \R^n \to \oR$ calm at each $x \in \R^n $ from below, the subderivative $\d f(x)(.)$ is proper and lower semicontinuous for each $x \in \R^n$. Therefore, the minimization  sub-problem in step 2 admits a solution.
\item[\textbullet] The rich calculus of the subderivative established in section \ref{sect02} eases the computation of $\d f$ for most functions, in particular, the non-amenable composition.  These calculus are not available for other primal tools such as Clarke directional derivatives. The latter is the prime reason we assign and specify the name   ``subderivative method" over ``a generalized gradient method". 
\item[\textbullet] The minimization sub-problem in step 2 does not need to be solved exactly, we only need to find an $w_k$ with $\d f(x_k)(w_k) < - \varepsilon .$ If $f$ is Clarke regular at $x_k$, then the sub-problems with subderivative method are convex programmings, indeed, in this case $\d f(x_k)(.)$ is a sub-linear function. In another case, if $f$ is differentiable at $x_k$ (not necessarly Clarke regular) the solution of the sub-problem in step 2 is $w_k =  - \frac{\nabla f(x_k)}{\| \nabla f(x_k) \|}$, and step 1 is simply checking if $\| \nabla f(x_k) \| \leq \varepsilon $. Luckily, we mostly deal with locally Lipschitz continuous functions; hence, it is likely that on many iterations we just have $w_k =  - \frac{\nabla f(x_k)}{\| \nabla f(x_k) \|}$, as Lipschitz continuous functions are differentiable almost everywhere by the Rademacher theorem.

\item[\textbullet] If $f$ is semi-differentiable, as it is in the all examples in section \ref{sect01}, then the subderivative method becomes a descent method; see Lemma \ref{backt}. In particular, in the constrained optimization \eqref{conso}, the penalized function $f(x) := \varphi (x) + \rho \; \dist(G(x) , X)$ is semi-differentiable, if $X$ is derivable, and $G$ and $\varphi$ are semi-differentiable; See Proposition \ref{semidiff-chain} and  Example \ref{semidist}. Hence, the subderivative method applied on $\varphi (x) + \rho \; \dist(G(x) , X)$ is a descent method.     
\end{itemize}  
}    
\end{Remark}    
\subsection{Convergence Analysis}

In order to establish the convergence of the subderivative method, we need that the generated directions (eventually) decrease the objective function sufficiently with a step size under consideration. Therefore, we propose the following condition.
\begin{Definition}[\bf descent property]\label{def_descent_pro}
Let $f : \R^n \to \oR$ and $\O \subseteq \R^n$. We say $f$ has a descent property on $\O$ if there exists a $L:=L(f , \O) \geq 0$ such that for all $x , y \in \O$ we have
\begin{equation}\label{descet_pro}
f(y) \leq f(x) + \d f(x) (y-x) + \frac{L}{2}  \| y -x  \|^2
\end{equation} 
\end{Definition}     
In the above definition, we do not require $f$ to be differentiable, not even Lipschitz continuous. If $f$ is continuously differentiable and its derivative is Lipschitz continuous, then \eqref{descet_pro} holds on $\O = \R^n$. The latter is known as the \textit{descent lemma}; see \cite[Lemma~4.22]{b}. An important class of non-amenable functions that enjoys the above descent property is the class of non-smooth concave functions. Furthermore, it is not difficult to verify that the descent property is preserved under addition and non-negative scalar multiplication. The composite function $f = g \circ F$ satisfies descent property \eqref{descet_pro}, if $g$ is concave, and $F$ is $L-$smooth function, i.e. $\nabla F(.)$ is Lipschitz continuous with a constant $L$. A striking observation is that the Moreau envelope of any bounded-below function enjoys the descent property \eqref{descet_pro} without assuming convexity or weak convexity. More precisely, if $f : \R^n \to \oR$ is prox-bounded with a constant $c > 0$, i.e., $f + \frac{1}{2c} \| . \|^2$ is bounded below \cite[Definition~1.23]{rw} , then $e_{r} f(x) := \inf_{y\in \R^n}\{ \frac{1}{2r} \|  x - y \|^2  + f(y) \}$ satisfies the descent property for all $r \in ( 0 , c )$.  We gather the above observations in the following proposition.
\begin{Proposition}[\bf sufficient conditions for the descent property]\label{sufdes}
Let $f, \; g: \R^n \to \oR$, $\O \subseteq \R^n$, and $F: \R^n \to \R^m$. Then, the following assertions hold
\begin{itemize}
\item[1.] if $f$ is $L$-smooth, then $f$ has the descent property  \eqref{descet_pro} on $\R^n$.
\item[2.] if $\lambda > 0$, and $f$ satisfies the descent property \eqref{descet_pro} on $\O$, then $\lambda f$ has the descent property on $\O$.
\item[3.] if $f$ and $g$ satisfy the descent propert on $\O$ so is $f + g$. 
\item[4.] if $g$ is a finite-valued concave function, then $g$ satisfies the descent property on $\R^n$ with the constant $L = 0$.
\item[5.]if $\O$ is a convex-compact set,  $F$ is $L$-smooth on $\O$,  and  $g$ is a finite-valued concave function, then $g \circ F$ has the descent property on $\O$ with a constant $L$.
\item[6.] if $f$ is prox-bouded with a constant $c$, then for all $r \in (0, c), $ $e_{r} f$ has the descent property with a constant $L = \frac{1}{r}$. 
\item [7.] if $\O$ is a convex-compact set,  $F$ is $L$-smooth on $\O$,  and  $g$  is prox-bounded with a constant $c$, then for each $r \in (0 , c)$, $[e_r g] \circ F$ has the descent property on $\O$ with a constant $\frac{L}{r} .$ 
\item [8.] Let $g_i : \R^{m_i} \to \oR $ be convex functions, and let $F_i: \R^{n} \to \R^{m_i}$ be $L$-smooth functions for each $i=1,2$, then for each $r \in (0 , \infty)$,  $f:= [\e_r g_1 ] \circ F_1   -   g_2  \circ F_2 $ has the descent property on $\R^n$  with a constant $\frac{L(1+r)}{r}.$  
\end{itemize}	
\end{Proposition}
\begin{proof}
(1) is a classical result known as the descent lemma \cite[Lemma~4.22]{b} . (2) is obvious. To prove (3), observe  from  descent property that $  - \infty < \d f(x)(w)$ for all $x , w \in \R^n$. Hence, by defnition of ``$\liminf$",  we  get $\d f(x)(w) + \d g(x)(w) \leq \d (f+g)(x)(w)$ for all $x , w \in \R^n $, without imposing any qualification condition for sum. Now the descent property for $f+g$ follows from adding the corresponging inequalities in \eqref{descet_pro} for  both $f$ and $g$. Turning to (4), note that finite-valued concave functions are locally Lipshitz continuous and directionally differentiable, thus they are semi-differentiable with $g' (x; w) = \d g(x)(w) $ for all $x ,w \in \R^n$.  Therfore, (4) immedietly follows from the inequality for concave functions 
$$ g(y) - g(x)  \leq g' (x  ; y -x) .$$
To prove  (5),  first we recall some properties of $\d g$. Ineed, since $g$ is concave,  $\d g(z)(.)$ is concave and positively homeogeneous for all $z \in R^m$. Hence, for all $a , b \in \R^m$ one has 
$\d g(z)(a) \leq \d g(z)(a+b) - \d g(z)(b)$. Moreover, $| \d g(z)(w)  |\leq \ell \|w\|$, where $\ell > 0$ is any Lipschitz constant of $g$ on any neighborhood of $z$. We use the latter facts shortly.  Now set $\eta (x , y) := F(y) - F(x) - \nabla F(x) (y -x)  $. Since $F$ is $L$-smooth on $\O$, we have $\|\eta  (x ,y)\|  \leq  \frac{ L}{2}  \; \|y -x\|^2 $ for all $x ,y \in \O$. Now  fix $x ,y \in  \O$, by concavity of $g$ we have 
\begin{eqnarray*}
	\disp
	g(F(y))& \leq &  g(F(x)) + \d g(F(x))(F(y)-F(x))\\\nonumber
	&=&g(F(x)) + \d g(F(x))\big(\nabla F(x) (y -x) + \eta (x , y)\big)\\\nonumber
	&\leq&g(F(x)) + \d g(F(x))\big(\nabla F(x) (y -x)\big) -  \d g(F(x)) \big(-\eta (x , y) \big)\\\nonumber
	&\leq& g(F(x)) +  \d g(F(x))\big(\nabla F(x) (y -x)\big)  + \ell \|  - \eta (x,y)  \|  \\\nonumber
	&\leq& g(F(x)) + \d (g  \circ F) (x)(y -x) + \frac{\ell L}{2} \|  y - x \|^2  \\\nonumber
	\disp
\end{eqnarray*}
where $\ell > 0$ is any Lipschitz constant of $g$ over the set $\O + \B$.  Note that in the third and forth inequalities we used  the some properties of $\d g$ we discussed earlier. In the last inequality, we used the chain rule \eqref{fcalc1} .  Part (6), immediately follows from part (3) and (4) and the fact that the Moreau envelope of any function is a sum of a $L$-smooth function and a concave function. Indeed, $e_r f (x)  > - \infty$  for all $r \in (0 , c)$, and we have 
$$    e_r f (x) = \frac{1}{2r}    \|   x \|^2   + \inf_{y}   \{ - \frac{1}{r}   \la x , y \ra + \frac{1}{2r}   \|y\|^2  + f(y)  \} $$ 
noting that the infimum over an arbitrary number affine functions produces a concave function. To prove (7), we write $ e_r g (z) = \varphi (z) + h(z)$, where $\varphi$ is a $\frac{1}{r}$-smooth function and $h$ is a concave function. 
Hence, we have $([e_r g] \circ F) (x) = (\varphi \circ F)(x)  +  (h \circ F)(x) $. Observe that $\varphi \circ F$ is $\frac{L}{r}$-smooth function, and $h \circ F$ satisfies the descent property by (5), thus (7) follows from (2). Assertion (8) follows from $(7)$, $(5)$, and $(3).$  
\end{proof}
Observe that in (5), the compactness of  $\O$  is redundant if $g$ is Lipschitz continuous on $F(\O)$. It is well-known that the Moreau envelope of convex functions is $L$-smooth. Thus (6) can be read as a far generalization of the smoothening results to any prox-bounded functions, in particular, functions that are bounded from below. Assertion (8) pertains to the difference of amenable functions, in particular, the difference of convex functions when $F_i$ are taken identity maps; see Example \ref{dame}. The difference of amenable functions contains a broad class of non-amenable functions. In \cite{m22}, we show that the sub-problems of the subderivative method can be accurately solved for some interesting sub-class of difference of amenable functions, in particular, \eqref{dmax}. As we will see later, the convergence of the subderivative method solely depends on the descent property, thus applicable to the Moreau envelope of any bounded-below function, no matter if a function is highly discontinuous or irregular. For instance, the Moreau envelope of the norm-zero function, which has a closed-form; see Example \ref{spars}. The following result shows a rudimentary convergence of the subderivative method with a pre-defined line search. 
\begin{Lemma}[\bf convergence of subderivative method with a diminishing line search]\label{Th_easyconve}
Suppose $f: \R^n \to \oR$ satisfy the descent property \eqref{descet_pro} on $\R^n$. Then, the subderavative method with the positive tolerance $\varepsilon > 0$ and the step-size $\alpha_k > 0$ defined by any sequence $\{\alpha_k\}_{k=0}^{\infty}$ satisfying $\sum_{k} \alpha_k = \infty ,$ and $\sum_{k} \alpha_k^2 < \infty $, either stops after finite iterations, or it generates a sequence $\{x_k\}_{k=0}^{\infty}$ for which $f(x_k) \to -\infty$.
\end{Lemma}
\begin{proof}
We suppose the subderivative method does not stop at any iteration, meaning that for the generated vector $x_k$, the corresponding descent direction satisfies $\d f(x_k) (w_k) < - \varepsilon .$ Plugging $x_{k+1} = x_k + \alpha_k w_k$ into \eqref{descet_pro} we get
\begin{eqnarray*}
\disp
f(x_{k+1}) - f(x_k)  &\leq &  \d f(x_k) (x_{k+1} - x_k) + \frac{L}{2}  \| x_{k+1}   -  x_k \|^2 \\\nonumber
&=& \alpha_k \d f(x_k)(w_k)  + \frac{ L}{2} \alpha_k^2\\\nonumber
&\leq & - \varepsilon  \alpha_k  + \frac{L}{2} \alpha_k^2 \\\nonumber
\disp
\end{eqnarray*}
By summing up both sides of the above inequality over $k=0,1, ...$ we get
$$ \lim_{k \to \infty }f(x_{k+1}) - f(x_0) \leq - \varepsilon  \sum_{k=0}^{\infty}\alpha_k  + \frac{L}{2}  \sum_{k=0}^{\infty} \alpha_k^2  = - \infty $$
This finishes the proof. 
\end{proof}
\begin{Remark}[\bf Armijo backtracking line search] {\rm
As the proof of Lemma \ref{Th_easyconve} confirms, in the subderivative method, one does not need to solve the minimization sub-problems accurately. Any direction $w_k$  with $\d f(x_k)(w_k) < -\varepsilon$ will do the job. Although the step-size chosen in Lemma \eqref{Th_easyconve} is simple and can be determined independently from $x_k$ and $w_k$, it may cause slow convergence. One of the most efficient and practically used line searches is \textit{Armijo backtracking} method which can be adapted to the subderivative version; indeed, fix the parameter $\mu \in (0 , 1)$, called a reduction multiple, and assume that we are in the $k^{th} -$iteration. The Armijo backtracking line search determines the step-size $\alpha_k >0$ in the following way:
If the following inequality holds for $\alpha_k = 1$, then the step-size is chosen $\alpha_k =1$.  
\begin{equation}\label{backtrine}
f(x_k + \alpha_k w_k) - f(x_k) <\frac{\alpha_k}{2}  \d f(x_k) (w_k).
\end{equation}
Otherwise, keep updating $\alpha_k$ by multiplying it by $\mu$ until the above inequality holds. In the next lemma, we show that backtracking method terminates after finite numbers of updates onder some mild assumptions.}
\end{Remark}
\begin{Lemma}[\bf finitness of Armijo backtracking process]\label{backt}
In the each of the following situations,  the Armijo backtracking process to find the step-size $\alpha_k$, within the subderivative algorithm, terminates after finite number of iterations.
\end{Lemma}
\begin{itemize}
\item[(i)] $f$ is  directionally lower regular (in particular, semi-differentiable) at $x_k$ for each $k=1,2...$
\item[(ii)] $f$ satisfies the descent property \eqref{descet_pro} on $\R^n .$
\end{itemize}
\begin{proof}
Let $x_k , w_k \in \R^n$ with $\d f(x_k) (w_k) < 0$. Hence, $\d f(x_k) (w_k) < \frac{1}{2} \d f(x_k) (w_k).$ In situation (i), we have 
$$  \d f(x_k) (w_k) = \lim_{ t \dn 0} \frac{f(x_k + t w_k) - f(x_k)}{t} < \frac{1}{2} \d f(x_k) (w_k) $$
hence, there exists $m \in \N$ sufficiently large such that
$$  f(x_k + \mu^{m} w_k) - f(x_k) <\frac{\mu^{m}}{2}  \d f(x_k) (w_k) .$$
In situation (ii), for all $t >0 $, the descent property \eqref{descet_pro} gives us 
\begin{eqnarray*}
\disp
\frac{f(x_k + t w_k) - f(x_k)}{t} \leq  \d f(x_k)(w_k)  + \frac{ L}{2} t .\\\nonumber
\disp
\end{eqnarray*}
Define $m$ by the smallest non-negative integer number such that $$   \d f(x_k)(w_k)  + \frac{ L}{2} \mu^m < \frac{1}{2} \d f(x_k)(w_k) $$  
where $\mu \in (0 ,1)$ is the reduction multiple in the Armijo backtracking method. Thus, it can be observed that \eqref{backt} holds for $\alpha_k = \mu^{m}$.
\end{proof}
In the following theorem, we prove that the subderivative method with Armijo backtracking line search,  intialized on an aribitrary vector in $\R^n$, finds a $\varepsilon$-$d$-stationary point after atmost $O(\varepsilon^{-2})$ iterations.  
\begin{Theorem}[\bf convergence rate of the subderivative method]\label{rate_TH}
Let $\{x_k\}_{k =0}^{\infty}$ be a sequence generated by the subderivative method with the Armijo backtracking line search. Assume that $f$ is bounded below and has the descent property \eqref{descet_pro} on $\R^n$. Then, for sufficiently large number $N \in \N$, the following inequality holds
	\begin{equation}\label{rate}
		\min_{ 0 \leq k \leq N}  |\d f(x_k)(w_k) |   \leq \sqrt{\frac{f(x_0) - f^{*}}{M (N+1)}}
	\end{equation}
	where $f^*$ is a lower bound of $f$ and $M:=\min\{\frac{1}{2} , \frac{\mu}{2L}\} $.
\end{Theorem}
\begin{proof} Before proving \eqref{rate}, let us first prove the following claim that guarantees the sufficient decrease of the subderivative method under the descent propety assuption. 
\textbf{Claim:} For each $k\in \N \cup \{0\}$,  the follwing inequality holds
\begin{equation}\label{suf-dec1}
f(x_{k+1}) - f(x_k)   \leq  -  M  \;  \min\{ |\d f(x_k)(w_k)| \; , \;  [\d f(x_k)(w_k)]^2 \}
\end{equation}
proof: set  $\delta_k := \min\{ |\d f(x_k)(w_k)| , [\d f(x_k)(w_k)]^2 \}$. We suppose we are in the $k^{th}-$ iteration of the subderivative method and moving toward to the next generated point. This ensures the existence of the direction $w_k \in \R^n$ with $\d f(x_k)(w_k) < 0$. By the descent property \eqref{descet_pro} for all $t >0$ we have
\begin{eqnarray}\label{sufeq1}
	\disp
	f(x_{k} + t w_k) - f(x_k)  &\leq &  \d  f(x_k) (tw_k) + \frac{Lt}{2}  \| w_k \|^2 \\\nonumber
	&=& t \d f(x_k)(w_k)  + \frac{ L t^{2}}{2} 
	\disp
\end{eqnarray} 
Let $\alpha_k$ be determined by the backtracking procedure, meaning that the stepsize $\frac{\alpha_k}{\mu}$ is not acceptable and does not satisfy \eqref{backtrine}:
\begin{equation}\label{sufdec3}
	f(x_k + \frac{\alpha_k}{\mu} w_k ) - f(x_k)  >  \frac{\alpha_k}{2\mu} \d f(x_k)(w_k).
\end{equation}
Combining the above inequality with \eqref{sufeq1} we get
\begin{eqnarray*}
	\disp
	& &  \frac{\alpha_k}{2\mu} \d f(x_k)(w_k) < \frac{\alpha_k}{ \mu} \d f(x_k)(w_k)  + \frac{ L \alpha_k^{2}}{ 2\mu^2}\\\nonumber
	&\rightarrow &   \d f(x_k)(w_k) < 2 \d f(x_k)(w_k)  + \frac{ L \alpha_k}{ \mu}\\\nonumber
	&\rightarrow &  - \frac{ \mu \d f(x_k)(w_k)}{L} < \alpha_k \\\nonumber 
	&\rightarrow &  -\frac{\mu}{2L} [\d f(x_k)(w_k)]^2 > \frac{\alpha_k}{2} \d f(x_k)(w_k). \\\nonumber
	\disp
\end{eqnarray*}       
Combining the above last inequality with \eqref{backtrine} we get
$$   f(x_k + {\alpha_k} w_k ) - f(x_k)  \leq  \frac{\alpha_k}{2} \d f(x_k)(w_k) < -\frac{\mu}{2L} [\d f(x_k)(w_k)]^2 \leq -M \delta_k .$$     
This completes the proof of the claim. Turning to the proof of \eqref{rate}, by summing up the both sides of \eqref{suf-dec1} we get
\begin{eqnarray}\label{conv1}
	\disp
	&& M \sum_{k=0}^{N} \delta_k  \leq  \sum_{k=0}^{N} f(x_k) - f(x_{k+1})\\\nonumber
	&\rightarrow&  (N+1)M \min_{ 0 \leq k \leq N} \{\delta_k \} \leq f(x_0) - f(x_{N+1})  \leq  f(x_0) - f^*\\\nonumber  
	& \rightarrow & \min_{ 0 \leq k \leq N} \{\delta_k\}  \leq \frac{f(x_0) - f^*  }{M (N+1)}
	\disp
\end{eqnarray}
Above inequality holds for all non-negative integer $N$. Additionally,  if $N$ is sufficiently large, we have $\min_{ 0 \leq k \leq N} \{\delta_k\} < 1$. Considering the definition of $\delta_k$, for such $N$ we have
$$    \min_{ 0 \leq k \leq N} \{\delta_k\} = \min_{ 0 \leq k \leq N} \{  [\d f(x_k)(w_k)]^2 \}    $$
this together with last inequality obtained in \eqref{conv1} proves \eqref{rate}.
\end{proof}
In Theorem \eqref{rate_TH}, we really do not need to impose the descent property on entire $\R^n$. For instance, if $f$ is semi-differentiable and coercive,  we can  assume the descent property locally on $\R^n$. In the following theorem, similar to the gradient descent method for smooth minimization,  we show that we only need to have the descent property on a sublevel set containing the iterations of the algorithm. 
\begin{Theorem}[\bf convergence of subderivative descent method]\label{rate_TH2}
Let $\{x_k\}_{k =0}^{\infty}$ be a sequence generated by the subderivative method with the Armijo backtracking line search. Set $\O := \{ x \big| \; f(x) \leq f(x_0)\} $. Assume that $f$ is bounded-below, semi-differentiable on $\O$, and has the descent property \eqref{descet_pro} on $\O + B$. Then, for sufficiently large number $N \in \N$, \eqref{rate} holds.  
\end{Theorem}
\begin{proof}
First, we prove inequality \eqref{suf-dec1} when the descent property \eqref{descet_pro} is only assumed on $\O + B.$ To prove the latter, by Lemma \eqref{backt} (i), the subderivative method is a descent method, because $f$ is semi-differentiable on $\O$. Hence, for each $k \in \N$, we have  $x_k \in \O$.  Indeed,  since $f$ is semi-differentiable, for each $ k \in \N$, we have $\d f(x_k)(w_k) = f' (x_k ; w_k) <0 $. Therefore, for each $k \in \N$, there exists a $\varepsilon_k >0 $ such that $x_k + t w_k \in \O$ for all $t \in [0, \varepsilon_k]$, hence, $x_k + t w_k \in \O + \B$ for each $t \in [0, 1 + \varepsilon_k]$. Now proceeding with the proof of inequality \eqref{suf-dec1}, inequality \eqref{sufeq1} holds for all $t \in [0, 1 + \varepsilon_k]$. Now if $\alpha_k \neq 1$ is determined by the Armijo backtracking procedure, inequality \eqref{sufdec3} holds with $\alpha_k = \mu^m$ for some positive integer $m$. Since $ \frac{\alpha_k}{\mu} \in (0 ,1 + \varepsilon_k]$, we get $x_k + \frac{\alpha_k}{\mu} w_k $. Form here, the rest of the proof, including the proof of \eqref{rate},  agrees line by line of the proof of Theorem \ref{rate_TH} after inequality \eqref{sufdec3}. 
\end{proof}
\begin{Remark}[\bf comments on convergence results]{\rm
Since norms are equivalent in finite dimensions, the descent property \eqref{descet_pro} is invariant under the choice of the norm, i.e.,  \eqref{descet_pro}  holds if and only it holds with a quadratic term $\|y-x\|_{*}^2$ where $\|  . \|_{*}$ is any norm on $\R^n$. Therefore, if in the subderivative method, one solves the sub-problems with respect to any norm $\|.\|_{*}$, i.e., $\| w\|_{*} \leq 1$, the convergence results, Theorems \ref{rate_TH} and \ref{rate_TH2},  remain true. The only thing that will change is the constant $M$ in \eqref{rate}. As we will see shortly, changing the norm sometimes leads to solving the sub-problems of the subderivative method accurately. A close look at the proof of Theorem \ref{rate_TH2}, reveals that the assumption of semi-differentiability can be replaced by directional lower regularity, as we only used the $\d f(x)(w) = f' (x;w)$ as result of semi-differentiability. The latter improvement might seem minor, yet there are functions which are directionally lower regular but not semi-differentiable; e.g. $f(x) = \| x \|_{0}$, Example \ref{subzero}.}
\end{Remark}

\begin{Example}[\bf problems with separable-variable subderivative]\label{sep} {\rm Let $\varphi : \R^n \to \R$ be a differentiable function, and $ g : \R^n \to \R$ be a function with separable-variable subderivative, i.e, there exist one-variable functions $g_i : \R \to \R$ , $i =1,...,n$ such that for each $x \in \R^n$ and each $w \in \R^n$ one has $ \d g(x)(w) = \sum^{n}_{i=1} g_i (w_i ) $. In this example, we consider the following minimization problem:
		\begin{equation}\label{sep1}
			\mbox{minimize} \;\;  f(x):= \varphi (x) + g(x)  \quad \mbox{subject to}\;\; x \in \R^n.
		\end{equation}
		For each $x \in \R^n$ and each $i \in \{ 1,...,n \}$, let $s_i$ and $w_i$ be the optimal value and an optimal solution of the following one-variable minimization problem respectively
		\begin{equation}\label{sep2}
			P(x , i): \quad \quad \mbox{minimize} \;\;  \frac{\partial \varphi (x)}{ \partial x_i}  t + h_i(t)  \quad \mbox{subject to}\;\; t \in [-1 , +1].
		\end{equation}
		Then, the subderivative method for problem \eqref{sep1}, with the norm-choice $\| . \|_{\infty}$ in sub-problems, reduces to the following process
		\begin{table}[H]
			\renewcommand{\arraystretch}{0}
			\centering
			\label{alg1}
			\begin{tabular}{|
					p{0.9\textwidth} |}
				\hline
				\textbf{}\\
				
				\begin{enumerate}[label=(\alph*)]
					
					\item[0.] \textbf{(Initialization)} Pick the tolerance $\varepsilon > 0$, starting point $x_0 \in \R^n$, and set $k=0 .$
					
					\item[1.] \textbf{(Termination)} $\sum^{n}_{i=1} s^k_i \geq - \varepsilon $ where each $s^k_i$ is optimal value of the problem $P(x_k , i)$ in \eqref{sep2}.
					\item[2.] \textbf{(Direction Search)} The components of the descent direction $w_k$ are $w_k^i$, an optimal solution of problem $P(x_k , i)$. 
					\item[3.] \textbf{(Line Search)} Choose the step size $\alpha_{k}  > 0$ through a line search method.
					\item[4.] \textbf{(Update)} Set $x_{k+1} : = x_k + \alpha_k  w_k$ and $k+1  \leftarrow k$ then go to step 1. 
				\end{enumerate}
				\\ \hline
			\end{tabular}
		\end{table}  
	} 
\end{Example}
Detail for the direction search: given the iteration $x_k$, the subderivative of $f$ at $x_k$ in direction $w$ is $$ \d f(x_k)(w) = \la \nabla \varphi (x) , w  \ra + \d g(x_k) (w) = \sum_{i = 1}^{n} \big( \frac{\partial \varphi (x_k)}{ \partial x_i}  w_i + g_i(w_i) \big)  .$$ 
On the other hand, the constraint $\| w \|_{\infty} \leq 1$ is fully separable, and it is equivalent to $w_i \in [-1  , +1]$ for all $i =1,...,n$. The latter makes the sub-problem of direction search fully separable in $w$. Indeed,
$$  \min_{\| w \|_{\infty} \leq 1} \d f(x_k)(w) =\sum_{i = 1}^{n} \big( \min_{| w_i |\leq 1} \frac{\partial \varphi (x_k)}{ \partial x_i}  w_i + g_i(w_i) \big) = \sum_{i = 1}^{n} s_i^k .$$
Examples of non-differentiable functions with separable-variable subderivative are $g(x) = \| x\|_{1}$ and $g(x) = - \|x\|_{1}$. Indeed, in the case $g(x) =  \lambda \|  x\|_1$ where $\lambda >0$,  we have 
$$ \d g(x)(w) = \sum_{i \in I_1 \cup I_2} \lambda w_i + \sum_{i \in I_3 \cup I_4} (-\lambda w_i)  +  \sum_{i \in I_5 \cup I_6 \cup I_7} \lambda |w_i| $$
where
$$I_1 (x) = \{ i| \; x_i >0 \; ,\; \frac{\sub \varphi}{\sub x_i} (x) + \lambda \geq 0  \}, \quad I_2 (x) =\{ i| \; x_i >0 \; ,\; \frac{\sub \varphi}{\sub x_i} (x) + \lambda < 0  \}, $$

$$ I_3 (x) = \{ i| \; x_i < 0 \; ,\; \frac{\sub \varphi}{\sub x_i} (x) - \lambda \geq 0  \}, \quad I_4 (x) = \{ i| \; x_i < 0 \; ,\; \frac{\sub \varphi}{\sub x_i} (x) - \lambda < 0  \}$$

$$ I_5 (x) =\{ i| \; x_i = 0 \; ,\; \frac{\sub \varphi}{\sub x_i} (x) - \lambda \geq 0  \}, \quad I_6 (x) = \{ i| \; x_i = 0 \; ,\; \frac{\sub \varphi}{\sub x_i} (x) + \lambda \leq 0  \}$$

$$  I_7 (x) = \{ i| \; x_i = 0 \; ,\; \frac{\sub \varphi}{\sub x_i} (x) + \lambda > 0, \quad  \frac{\sub \varphi}{\sub x_i} (x) - \lambda < 0 \}  $$
Therefore, by Example \ref{sep}, the vector $w_k=(w_k^1,w_k^2,...,w_k^n)$ defined by
$$  w_{k}^i := \begin{cases}
	-1 &  i \in I_1 \cup I_3 \cup I_5 \\
	1&  i \in I_2 \cup I_4 \cup I_6 \\
	0&  i \in I_7
\end{cases} $$
solves the sub-problems of the subderivative method.  
\begin{Example}[\bf a sub-class of difference of amenable functions]\label{cc} {\rm Let $\varphi : \R^n \to R$ and $F:\R^n \to \R^m$ be $L$-smooth functions. Let $g: \R^m \to \R$ be a convex function. In this example, we consider a non-regular composite minimization problem in the following framework 
		\begin{equation}\label{cc1}
			\mbox{minimize} \;\;  f(x):= \varphi (x) - (g \circ F)(x)  \quad \mbox{subject to}\;\; x \in \R^n.
		\end{equation} 
For each $i \in \{1,2,...,n\}$, let $e_i \in  \R^n$ be the vector whose $i$-th component is $1$ and zero for other components. Define the finite set $E \subset \R^n$ with $2n$ elements by $E:=\{\pm e_i |\; i=1,2,...,n\}.$ Then, the subderivative method, with the norm-choice $\|.\|_{1}$, for problem \eqref{cc1} reduces to the following process
	\begin{table}[H]
	\renewcommand{\arraystretch}{0}
	\centering
	\label{alg1}
	\begin{tabular}{|
			p{0.9\textwidth} |}
		\hline
		\textbf{}\\
		
		\begin{enumerate}[label=(\alph*)]
			
			\item[0.] \textbf{(Initialization)} Pick the tolerance $\varepsilon > 0$, starting point $x_0 \in \R^n$, and set $k=0 .$
			
			\item[1.] \textbf{(Termination)}  $$s_k :=\min \{ \pm \frac{\sub \varphi (x_k)}{x_i} \mp  \d g(F(x_k))(\nabla F(x_k)e_i)  | \; i=1,2,...,n   \}  \geq - \varepsilon$$.
			\item[2.] \textbf{(Direction Search)} Pick $w_k \in E$ with the minimum value 
			$$  \la \nabla \varphi (x_k) , w_k \ra -\d g(F(x_k))(\nabla F(x_k)w_k)  $$
			\item[3.] \textbf{(Line Search)} Choose the step size $\alpha_{k}  > 0$ through a line search method.
			\item[4.] \textbf{(Update)} Set $x_{k+1} : = x_k + \alpha_k  w_k$ and $k+1  \leftarrow k$ then go to step 1. 
		\end{enumerate}
		\\ \hline
	\end{tabular}
\end{table}  
}
\end{Example}
Detail for the direction search: given the iteration $x_k$, the subderivative of $f$ at $x_k$ in direction $w$ is $$ \d f(x_k)(w) = \la \nabla \varphi (x_k) , w\ra -\d g(F(x_k))(\nabla F(x_k)w)  .$$
Observe that $\d f(x_k)(.)$ is a concave function, thus at least one of its global minimizers subject to the polytube $P:= \{w| \; \|w\|_{1} \leq 1 \}$ happens on an extreme point of $P$. The set of all extreme points of $P$ is $E$ defined above. Therefore, the sub-problems of the subderivative method in the direction search are accurately solved by the procedure in the above table. 
\begin{Remark}[\bf  difference of amenable functions]{\rm In Example \ref{cc}, the sub-problems of the subderivative method reduced to the evaluating of the subderivative function, $\d f(x_k)(.)$, at $2n$ elements of the set $E$, where $n$ is the number of variables. It is worth to mention that we could reduce these $2n$ evaluations to $n+1$ evaluations by evaluating $\d f(x_k)(.)$ at vectors $\{e_i |\; i=1,...,n\} \cup \{\ - e \} $ where $e \in \R^n$ is the vector with its components equals to $1$. Indeed,  $P:= \co \big(\{e_i |\; i=1,...,n\} \cup \{\ - e \} \big) $ is a convex polytube in $\R^n$ having origin in its interior. Therefore, there exists a norm $\|.\|_{*}$ such that $P=\{w| \; \| w \|_{*} \leq 1   \}$. Now it is enough to consider the subderivative method with the norm-choice $\|.\|_{*}.$ Although Example \ref{cc} deals with a sub-class of the difference of amenable functions, it can be used to solve all difference of amenable functions. Indeed, by applying a Moreau envelope on the convex part, the general case can be fitted into the framework \ref{cc1}. More precisely, in order to solve
   \begin{equation*}\label{cc2}
	\mbox{minimize} \;\;  (g_1 \circ F_1)(x) - (g_2 \circ F_2)(x)  \quad \mbox{subject to}\;\; x \in \R^n,
\end{equation*}  		   
one can apply a Moreau envelope only on the convex function $g_1$ to smoothening $g_1 \circ F_1$, i.e.,
  \begin{equation*}\label{cc3}
	\mbox{minimize} \;\;  \varphi(x) - (g_2 \circ F_2)(x)  \quad \mbox{subject to}\;\; x \in \R^n,
\end{equation*}  	
where $\varphi := [e_r g_1] \circ F_1$. The convergence analysis of the latter problem was discussed in Proposition \ref{sufdes} (8). In \cite{m22}, we investigate the difference of amenable functions in more detail and establish a relationship between the stationary points of the last two problems.  
}
\end{Remark}    

\textbf{Concluding remarks:} In this paper, we studied first-order variational analysis of non-regular functions, mainly the non-amenable composite functions. The approach was taken in this paper is purely primal. By establishing rich calculus rules for a suitable notion of directional derivatives, we could identify the stationary points in a large class of optimization problems, including the ones in the example section. We adopted a generalized gradient descent method to search for such stationary points. In particular, we showed that our algorithm applied on the Moreau envelope of any bounded-below function finds a $\varepsilon$-stationary point with the same rate that the smooth gradient descent does. Our future goal is to treat some examples in section \ref{sect01} numerically and solve the sub-problems of the subderivative method for some class of functions. In \cite{m22}, we give a full treatment of the numerical analysis of Example \ref{dmax}.    


\small
\begin{thebibliography}{10}
\bibitem{asp} { A. Alvarado, G. Scutari, and J. Pang}, {\em A new decomposition method for multiuser dc-programming and its applications}, IEEE Transaction on signal processing,  62(2014),  no. 11.

\bibitem{b} {A. Beck}, {\em Introduction to nonlinear optimization}, SIAM, Philadelphia, 2014.

\bibitem{by} { K. Bui and J. Ye}, {\em Directional necessary optimality conditions for bilevel programs}, Math. Ope. Res. (2021), published Electronically.

\bibitem{byz} { K. Bui, J. Ye, and J.Zhang}, {\em Directional quasi-/pseudo-normality as sufficient conditions for metric subregularity}, SIAM J. Optim., 29(2019), pp. 2625–2649.

\bibitem{bs} { J.F. Bonnans and A. Shapiro}, {\em Perturbation Analysis of Optimization Problems}, Springer, New York, 2000.

\bibitem{bl}{J. V. Burke and Q. Lin},{\em Convergence of the gradient sampling algorithm on directionally lipschitz functions}, 2021, arXiv 2107.04918.

\bibitem{bp}{ J. V. Burke and R. A. Poliquin }, {\em Optimality conditions for non-finite valued convex composite functions}, Math. Program, 57(1992), pp. 103-120.

\bibitem{bf} { J. V. Burke and M. C. Ferris}, {\em A Gauss-Newton method for convex composite optimization}, Math. Program., 71 (1995), pp. 179-194.

\bibitem{blo} { J. V. Burke, A. S. Lewis, and M. L. Overton}, {\em A robust gradient sampling algorithm for nonsmooth, nonconvex optimization}, SIAM J. Optim., 15(2005), 751–779. 

\bibitem{c} { F. Clarke F}. {\em Functional analysis, calculus of variations and optimal control}, Springer, London, 2013.

\bibitem{cp} { Y. Cui and J-S Pang}, {\em Modern nonconvex nondifferentiable optimization}, SIAM, Philadelphia, 2021.

\bibitem{cps} {Y. Cui, J-S Pang, and B. Sen}, {\em Composite difference-max programs for modern statistical estimation problems}, SIAM J. Optim., 28(2018), pp. 3344–3374.

\bibitem{dp} {D. Drusvyatskiy and C. Paquette}, {\em Efficiency of minimizing compositions of convex functions and smooth maps}, Math. Program., 178(2019), pp. 503-558.

\bibitem{dil} {D. Drusvyatskiy, A. D. Ioffe, and  A. S. Lewis}, {\em Nonsmooth optimization using Taylor-like models: error bounds, convergence, and termination criteria}, Math. Program., 185(2021), pp. 357-383.

\bibitem{ddmp} {D. Drusvyatskiy, D. Davis, K. J.  MacPhee, and C. Paquette}, {\em Subgradient methods for sharp weakly convex functions}, J. Optim. Theory Appl., 179(2018), pp. 962-982.

\bibitem{g} {H. Gfrerer}, {\em On directional metric regularity, subregularity and optimality conditions for nonsmooth mathematical programs}, Set-Valued Var. Anal., 21(2013), pp. 151–176.

\bibitem{i} {A. D. Ioffe }, {\em Variational analysis of regular mappings: theory and applications}, Springer, Switzerland, 2017.

\bibitem{io} {A. D. Ioffe and J. V. Outrata}, {\em On metric and calmness qualification conditions in subdifferential calculus}, Set-Valued Var. Anal., 16(2008), pp. 199-227.

\bibitem{k} {K. C. Kiwiel}, {\em Convergence of the gradient sampling algorithm for nonsmooth nonconvex optimization}, SIAM J. Optim., 18(2007), pp. 379–388.

\bibitem{l} { G. Lan}, {\em An optimal method for stochastic composite optimization}, Math. Program., 133(2012), pp. 1-33.

\bibitem{lw} {A. Lewis and S. J. Wright}, {\em A proximal method for convex minimization}, Math. Program., 158(2016), pp. 501–546.

\bibitem{mins} {L. Minchenko and S. Stakhovski}, {\em Parametric nonlinear programming problems under the relaxed constant rank condition}, SIAM J. Optim., 21(2011), pp. 314–332.

\bibitem{ash} { A. Mohammadi}, {\em Variational analysis of composite optimization}, PhD dissertation, Wayne State University. 2020.

\bibitem{m22} { A. Mohammadi}, {\em Subderivative method: a generalized gradient descent method for solving subdifferentially non-regular problems}, in preparation.

\bibitem{mm21} { A. Mohammadi and B. Mordukhovich}, {\em Variational analysis in normed spaces with applications to constrained optimization}, SIAM J. Optim., 31(2021), pp. 569-603.

\bibitem{mms1} {A. Mohammadi, B. Mordukhovich, and M. E. Sarabi}, {\em Variational analysis of composite models with applications to continuous optimization}, Math. Oper. Res., (2021), arXiv:1905.08837.

\bibitem{mms2} {A. Mohammadi, B. Mordukhovich and M. E. Sarabi}, {\em  Parabolic regularity in geometric variational analysis}, Trans. Amer. Soc., 374(2021), pp.1711-17630.

\bibitem{ms} {A. Mohammadi and M. E. Sarabi}, {\em Twice epi-differentiability of extended real-valued functions with applications in composite optimization}, SIAM J. Optim., 30(2020), pp. 2379–2409.

\bibitem{m18} {\sc B. S. Mordukhovich}, {\em Variational analysis and applications}, Springer, Cham, Switzerland, 2018.

\bibitem{N} {Y. Nesterov}, {\em Lectures on convex optimization}, Springer, 2018.

\bibitem{o} {J. V. Outrata }, {\em On the numerical solution of a class of Stackelberg problems}, Zeitschrift fur Oper. Res., 34(1990), pp. 255–277.

\bibitem{pra} {J-S. Pang, M. Razaviyan, and A. Alvarado}, {\em Computing B-Stationary points of nonsmooth DC Programs}, Math. Ope. Res., 42(2017), pp. 95–118.

\bibitem{rw} { R. T. Rockafellar and  R. J-B. Wets}, {\em Variational analysis}, Grundlehren Series (Fundamental Principles of Mathematical Sciences), Springer, Berlin, 2006.

\bibitem{r88} { R. T. Rockafellar}, {\em First- and second-order epi-differentiability in nonlinear programming}, Trans. Amer. Math. Soc., 307(1988), pp. 75–108.

\bibitem{r89} { R. T. Rockafellar}, {\em Second-order optimality conditions in nonlinear programming obtained by way of epi-derivatives}, Math. Ope. Res. Soc., 14(1989), pp. 462-484.

\bibitem{r21} { R. T. Rockafellar}, {\em Convergence of augmented Lagrangian methods in extensions beyond nonlinear programming}, preprint. (2021). 

\bibitem{jzz} {J. J. Ye, D. L. Zhu, and Q. J. Zhu}, {\em Exact penalization and necessary optimality conditions for generalized bilevel programming problems}, SIAM J. Optim., 7(1997), pp. 481–507.

\bibitem{zl} {Z. Zhang, and G. Lan}, {\em Optimal algorithms for convex nested stochastic composite optimization}, 2020, arXiv: 2011.10076.

\bibitem{zljsj} {j. Zhang, H. Lin, S. Jegelka, S. Sra, and A. Jadbabaie}, {\em Complexity of finding stationary points of nonsmooth nonconvex functions}, 2021, arXiv:2002.04130v3.

\end{thebibliography}
\end{document}